\documentclass[11pt]{amsart}
\usepackage{amssymb}
\usepackage[hmargin=1.1in,vmargin=1.15in]{geometry}
\usepackage{xcolor}
\usepackage{tikz}
\usetikzlibrary{arrows.meta}
\usepackage{algorithm}
\usepackage{algpseudocode}
\usepackage[final,expansion=false]{microtype}
\usepackage{booktabs}
\usepackage{array}
\usepackage{cite}
\usepackage{hyperref}

\hypersetup{
  colorlinks=true,
  linkcolor=blue!45!black,
  citecolor=blue!45!black,
  urlcolor=blue!45!black
}

\newtheorem{theorem}{Theorem}[section]
\newtheorem{lemma}[theorem]{Lemma}
\newtheorem{proposition}[theorem]{Proposition}
\newtheorem{corollary}[theorem]{Corollary}
\newtheorem{conjecture}[theorem]{Conjecture}
\theoremstyle{plain}
\newtheorem{namedthm}{Theorem}

\newtheorem{namedcor}[namedthm]{Corollary}
\newtheorem{namedconj}[namedthm]{Conjecture}
\theoremstyle{definition}
\newtheorem{definition}[theorem]{Definition}
\newtheorem{remark}[theorem]{Remark}

\newcommand{\Sg}{S_g}
\newcommand{\MCG}{\mathrm{MCG}}

\newcommand{\Aut}{\mathrm{Aut}}
\newcommand{\sgn}{\mathrm{sgn}}
\newcommand{\Vstd}{V_{\mathrm{std}}}

\newcommand{\Cent}{\mathrm{Cent}}

\makeatletter
\renewcommand\subsubsection{\@startsection{subsubsection}{3}{\z@}{.5\linespacing\@plus.7\linespacing}{-.5em}{\normalfont\bfseries}}
\renewcommand\paragraph{\@startsection{paragraph}{4}{\z@}{.5\linespacing\@plus.7\linespacing}{-.5em}{\normalfont\bfseries}}
\makeatother

\begin{document}

\title[Minimally intersecting filling curve systems]{Asymptotic enumeration of minimally intersecting filling curve systems on closed surfaces}
\author{Sayantika Mondal}
\address{Department of Mathematics, Fordham University, 441 East Fordham Road, Bronx, NY 10458, USA}
\email{smondal6@fordham.edu}

\begin{abstract}
Let $\Sg$ be the closed oriented surface of genus $g \geq 2$, and let $\Gamma$ be a finite collection of closed curves on $\Sg$ that fills, in the sense that its complement is a union of disks. The total number of double points of such a $\Gamma$ in minimal position is at least $2g-1$; we call $\Gamma$ \emph{minimally intersecting} when equality holds. We prove that the number $\mathcal{N}_g$ of mapping class group orbits of minimally intersecting filling curve systems satisfies
\[
\mathcal{N}_g \,\sim\, \frac{16^g\,(2g)!}{64\,\pi\sqrt{2}\,g^3} \qquad (g \to \infty).
\]

We also study the single-curve subproblem, in which $\Gamma$ consists of a single component. Every such curve is in particular a filling system, so the corresponding count $\mathcal{N}_g^{(1)}$ satisfies $\mathcal{N}_g^{(1)} \leq \mathcal{N}_g$, which improves the upper bound previously known for single filling curves. 
Finally, we compute $\mathcal{N}_g$ and $\mathcal{N}_g^{(1)}$ exactly for small genus and conjecture an asymptotic growth rate.
\end{abstract}

\maketitle

\section{Introduction}\label{sec:intro}

The study of curves on surfaces, and of mapping class group orbits of such curves and curve systems, is a recurring theme in low-dimensional topology and Teichm\"uller theory. A closed curve $\gamma$ on the closed oriented surface $\Sg$ \emph{fills} $\Sg$ if every component of $\Sg \setminus \gamma$ is a disk. More generally, a finite collection $\Gamma = \gamma_1 \cup \cdots \cup \gamma_c$ of closed curves fills $\Sg$ if every component of $\Sg \setminus \Gamma$ is a disk. All closed curves are assumed to be primitive, that is, not freely homotopic to a power of another closed curve. Filling curves and multicurves arise in the construction of pseudo-Anosov mapping classes~\cite{Penner1988,Thurston1988}, in the realization of hyperbolic structures~\cite{ThurstonSpine}, and in the description of the moduli space of curves via dessins d'enfants~\cite{LandoZvonkin2004}.  They were also the focus of my PhD thesis~\cite{MondalThesis}, and this paper grew out of parts of it: the material of Sections~\ref{sec:prelim}--\ref{sec:lowgenus} appears there, and provides a more detailed exposition on the character theoretic background and computations.

Throughout, $\Gamma$ is in \emph{minimal position}: among multicurves in its free homotopy class (free homotopy of each component, plus the minimum number of intersection points among distinct components), the total number of double points is minimum. We write $\iota(\Gamma)$ for this minimum; it counts both self-intersections of each component and intersections between distinct components. An Euler-characteristic computation (Proposition~\ref{prop:euler}; see~\cite{Basmajian2013}) gives the lower bound $\iota(\Gamma) \geq 2g-1$ for any filling system on $\Sg$, with equality if and only if the complement is a single disk. The bound is attained already by single curves: Arettines~\cite{Arettines2015} constructed, in every genus $g \geq 2$, a filling closed curve with exactly $2g-1$ self-intersections. We call such $\Gamma$ \emph{minimally intersecting}.

\subsection{Two enumeration problems}

We study two related counting problems:
\begin{itemize}
\item[(M)] count $\MCG(\Sg)$-orbits of minimally intersecting filling curve \emph{systems} $\Gamma$ (any number of components $c \geq 1$);
\item[(S)] count $\MCG(\Sg)$-orbits of minimally intersecting filling \emph{single} curves $\gamma$ (the case $c = 1$).
\end{itemize}
The dictionary developed in Section~\ref{sec:encoding} associates to a filling system on $\Sg$ an oriented $4$-regular ribbon graph of genus $g$, and translates both problems into conditions on a pair of permutations. A unicellular ribbon graph corresponds to a filling system, and the number of components can be recovered from the cycle structure of the Petrie permutation (Definition~\ref{def:petrie}), so single curves are the systems satisfying one further condition on this permutation. Theorem~\ref{thm:dictionary} records both correspondences. Sections~\ref{sec:encoding}--\ref{sec:lowgenus} treat problem (M); Section~\ref{sec:singlecurves} treats problem (S).

\subsection{Main results}

We denote by $\mathcal{N}_g$ the number of $\MCG(\Sg)$-orbits of minimally intersecting filling curve systems on $\Sg$, with $c \geq 1$ components, and $\mathcal{N}_g^{(1)}$ for the count restricted to single curves; here $\MCG(\Sg)$ is the group of isotopy classes of orientation-preserving homeomorphisms of $\Sg$.
The first result classifies these orbits combinatorially: via the dictionary of Section~\ref{sec:encoding}, they correspond to simultaneous-conjugacy classes of pairs of permutations satisfying explicit cycle-type conditions. Throughout, $N := 8g-4$, and $\sigma, \alpha \in S_N$ are permutations of the half-edge set $[N]$ of cycle types $4^{2g-1}$ and $2^{4g-2}$ respectively (conjugacy classes $C_\sigma$, $C_\alpha$); $\varphi := \sigma\alpha$ is the \emph{face permutation} and $\tau := \sigma^2\alpha$ the \emph{Petrie permutation}. We write $C_N$ for the class of $N$-cycles and $C_{(N/2)^2}$ for the class of permutations with two cycles of length $N/2$.

\begin{namedthm}\label{thmA}\textup{(Theorem~\ref{thm:dictionary-restate})}
Let $g \geq 2$. Under the dictionary $\Gamma \mapsto (\sigma,\alpha)$ of Section~\textup{\ref{sec:encoding}}:
\begin{enumerate}
\item[\textup{(i)}] $\MCG(\Sg)$-orbits of minimally intersecting filling curve \emph{systems} on $\Sg$ are in bijection with simultaneous-conjugacy classes of pairs $(\sigma,\alpha) \in C_\sigma \times C_\alpha$ satisfying the \emph{face condition} $\varphi = \sigma\alpha \in C_N$;
\item[\textup{(ii)}] among these, the orbits of \emph{single} filling curves are exactly the pairs that also satisfy the \emph{Petrie condition} $\tau = \sigma^2\alpha \in C_{(N/2)^2}$, equivalently those for which $G^\ast := (\sigma^2,\alpha)$ is connected.
\end{enumerate}
In particular $\mathcal{N}_g$ and $\mathcal{N}_g^{(1)}$ are finite: $\mathcal{N}_g$ counts the classes satisfying the face condition, and $\mathcal{N}_g^{(1)}$ those satisfying both.
\end{namedthm}

\begin{namedthm}\label{thmB}\textup{(Theorem~\ref{thm:main-asymptotic})}
As $g \to \infty$,
\[
\mathcal{N}_g \;\sim\; \frac{16^g\,(2g)!}{64\,\pi\sqrt{2}\,g^3}.
\]
\end{namedthm}

For single curves we prove an upper bound and an exact reduction, and conjecture a precise asymptotic.

\begin{namedcor}\label{corC}\textup{(Corollary~\ref{cor:single-upper})}
For every $g \geq 2$, $\mathcal{N}_g^{(1)} \leq \mathcal{N}_g$, so
\[
\mathcal{N}_g^{(1)} \;\leq\; \frac{16^g\,(2g)!}{64\,\pi\sqrt{2}\,g^3}\,(1+o(1)).
\]
This improves the upper bound of Parsad and Sanki~\cite{ParsadSanki2022} for single filling curves; see also~\cite{AougabHuang2015,AougabMenascoNieland} for the related count of minimally intersecting filling pairs.
\end{namedcor}

\begin{namedconj}\label{conjD}\textup{(Conjecture~\ref{conj:single-curve})}
As $g \to \infty$,
\[
\mathcal{N}_g^{(1)} \;\sim\; \frac{16^g\,(2g)!}{256\,\sqrt{2\pi}\,g^{7/2}}, \qquad \frac{\mathcal{N}_g^{(1)}}{\mathcal{N}_g} \;\sim\; \frac14\sqrt{\frac{\pi}{g}} \;=\; \sqrt{\frac{\pi}{16g}}.
\]
\end{namedconj}

For the proof of Theorem~\ref{thmB}, counting orbits reduces to counting pairs $(\sigma,\alpha) \in C_\sigma \times C_\alpha$ with $\sigma\alpha$ an $N$-cycle, and the central quantity is the probability that $\sigma\alpha$ is an $N$-cycle when $(\sigma,\alpha)$ is drawn uniformly from $C_\sigma \times C_\alpha$. A parity constraint then forces $\sigma\alpha$ into a single coset, on which the $N$-cycles have density $2/N$; this gives a heuristic value $2/N$ for the probability and explains the leading order, but says nothing about the error. To turn the heuristic into a theorem one needs the exact value, and this is what the translation into the character theory of $S_N$ provides: the count becomes a character sum that, for an $N$-cycle, collapses to a short and explicit form, from which the leading term $2/N$ reappears from two distinguished characters and the remainder is shown to be $O(N^{-2})$. A separate argument (Section~\ref{ssec:rooting}) shows that systems with non-trivial automorphism are too few to affect the asymptotic. Combining these with Stirling's formula gives the constant $1/(64\pi\sqrt{2})$.

For Conjecture~\ref{conjD}, the single-curve count carries the extra Petrie condition $\tau = \sigma^2\alpha \in C_{(N/2)^2}$, equivalently that the auxiliary $2$-regular map $G^\ast=(\sigma^2,\alpha)$ be connected. The prediction rests on two facts. First, $\sigma^2$ is a uniformly random fixed-point-free involution when $\sigma$ is uniform (Lemma~\ref{lem:squaring-surjective}), so unconditionally $\tau$ is a product of two random involutions and the two-$(N/2)$-cycle type has the \emph{exact} probability $q_g = \tfrac{2^{N-1}(N/2-1)!\,(N/2)!}{N!}\sim\sqrt{\pi/(2N)}$ (Theorem~\ref{thm:two-matchings}); this is $\Theta(N^{-1/2})$, far larger than the $\Theta(N^{-2})$ that a uniform-on-cycle-types model would give, because the cycle type of two random matchings is concentrated on a few long cycles. Second, the face condition ties $\sigma$ to $\sigma^2$, but through the Gelfand pair $(S_N, H_{N/2})$ the conditional probability is exactly $P_g = q_g\,\rho_U/p_g$, where $\rho_U$ is a single explicit one-face probability (Theorem~\ref{thm:gelfand-reduction}); the question reduces to whether $\rho_U \sim p_g$, i.e. whether the one-face conditioning is asymptotically independent of the Petrie cycle type (Conjecture~\ref{conj:decoupling}). Section~\ref{sec:singlecurves} proves the reduction, gives the numerical evidence for the remaining limit, and explains the obstacle to proving it.

\subsection{Comparison with prior work}

Cahn, Fanoni, and Petri~\cite{CahnFanoniPetri2018} count $\MCG(\Sg)$-orbits of curves with at most $k$ self-intersections for fixed $k$ as $g \to \infty$, where the curves are far from filling; the present paper treats the opposite extreme $k = 2g-1$. Minimally intersecting filling pairs of simple closed curves were studied by Aougab and Huang~\cite{AougabHuang2015} and by Aougab, Menasco, and Nieland~\cite{AougabMenascoNieland}, and minimally intersecting filling curves are used as building blocks in~\cite{BasmajianMondal2025}. For single filling curves, Corollary~\ref{corC} improves the upper bound of Parsad and Sanki~\cite{ParsadSanki2022}. The enumeration of unicellular maps goes back to Harer and Zagier~\cite{HarerZagier1986}, with later treatments by Goulden--Jackson~\cite{GouldenJackson2008}, Goupil--Schaeffer~\cite{GoupilSchaeffer1998}, and Chapuy~\cite{Chapuy2011}; the connectivity of the auxiliary map $G^\ast = (\sigma^2,\alpha)$, which governs the single-curve count, has no counterpart there. Dually, a uniformly random pair $(\sigma,\alpha)$ is a random gluing of quadrilaterals, for which Chmutov and Pittel~\cite{ChmutovPittel2016}, extending Gamburd~\cite{Gamburd2006}, showed that $\sigma\alpha$ is within total variation distance $O(N^{-1})$ of the uniform distribution on its coset of $A_N$; since the $N$-cycles have probability of order $N^{-1}$, this does not determine $p_g$, which Theorem~\ref{thm:probability} computes up to $O(N^{-2})$.

\subsection{Organization}\label{ssec:organization}

Section~\ref{sec:prelim} provides the necessary topological background, a self-contained definition of ribbon graphs and their realization, the Petrie permutation, and a primer on the representation theory of $S_N$ (\S\ref{ssec:char-bkgd}). Section~\ref{sec:encoding} develops the curve-system--ribbon-graph dictionary and the systems bijection (Proposition~\ref{prop:bijection}), records the three-level correspondence (\S\ref{ssec:threelevel}), sets up the permutation encoding and proves Theorem~\ref{thmA}. Section~\ref{sec:probability} computes the unicellularity probability, and Section~\ref{sec:asymptotic} combines this with the orbit count, the automorphism estimate, and Stirling's formula to prove Theorem~\ref{thmB}. Section~\ref{sec:lowgenus} gives exact values in low genus; the enumeration algorithm used there is described in Appendix~\ref{sec:algorithm}. Section~\ref{sec:singlecurves} treats the single-curve subproblem: the upper bound, the unconditional Petrie law and the Gelfand-pair reduction, which are proved, and the resulting conjecture with its supporting data; the sampler behind the data is described in Appendix~\ref{app:sampler}.

\subsection*{AI reflections}
The single-curve problem arose naturally from my thesis work, and I had thought about it for some time, I wanted to implement the encoding of filling curves by pairs of permutations in $S_N$ and compute small genus values, so asked an AI assistant (Claude, by Anthropic) to write code computing small cases. In the process, it pointed out that the conditions I had given it described the problem for curve systems rather than for single curves, and it pointed me to results from the character theory of $S_N$, which while standard, are not tools I would have naturally reached for. This led me to study the problem for systems, which became the main results of this paper. It did not, however, solve the original problem: the single-curve count remains conjectural (Conjecture~\ref{conjD}). AI assistance was also used in writing the code in the ancillary files and in revising the exposition; all arguments have been written and checked by me, and I take full responsibility for the content.

\subsection*{Acknowledgments}
This paper grew out of the author's doctoral thesis at the Graduate Center, City University of New York. The author thanks her advisor, Ara Basmajian, for his guidance and support, and the Graduate Center for funding.

\section{Background and preliminaries}\label{sec:prelim}

Throughout, $\Sg$ denotes the closed oriented surface of genus $g \geq 2$, and $\MCG(\Sg)$ denotes its mapping class group of orientation-preserving homeomorphisms modulo isotopy~\cite{FarbMargalit2012}. A closed curve $\gamma$ is a continuous map $S^1 \to \Sg$ considered up to free homotopy; all closed curves are assumed to be primitive, i.e.\ not freely homotopic to a proper power of another closed curve. A \emph{curve system} (or \emph{multicurve}) is an unordered finite union $\Gamma = \gamma_1 \cup \cdots \cup \gamma_c$ of closed curves, considered up to free homotopy of each component. We say $\Gamma$ is \emph{generic} if it is a smooth immersion whose only multiple points are transverse double points. Every system is freely homotopic to a generic one, and the \emph{intersection number} $\iota(\Gamma)$ is the minimum total number of transverse double points over all generic representatives.

\subsection{Ribbon graphs}\label{ssec:ribbon}

Several inequivalent-looking definitions of ribbon graphs coexist in the literature, differing in whether faces are read off from $\sigma\alpha$ or $\sigma^{-1}\alpha$, in the orientation convention, and in whether connectivity is built in. Because the present paper depends on the precise interaction between the cycle structures of $\sigma$, $\alpha$, and several derived permutations, we fix one convention and develop everything from it, so that the reader need  not refer to other sources.

\begin{definition}\label{def:ribbon}
A \emph{ribbon graph} (also \emph{oriented combinatorial map}) is a triple $G = (H,\sigma,\alpha)$ where:
\begin{itemize}
\item $H$ is a finite set, whose elements are \emph{half-edges};
\item $\alpha$ is a fixed-point-free involution on $H$, whose orbits are the \emph{edges} of $G$;
\item $\sigma$ is a permutation of $H$, whose orbits (each with the cyclic ordering induced by $\sigma$) are the \emph{vertices} of $G$.
\end{itemize}
The \emph{face permutation} is $\varphi := \sigma\alpha$; its orbits are the \emph{faces} of $G$. We say $G$ is \emph{$4$-regular} if every $\sigma$-orbit has length $4$, and \emph{unicellular} if $\varphi$ is a single cycle. Here and throughout, a product $\sigma\alpha$ of permutations means ``apply $\alpha$ first, then $\sigma$''.
\end{definition}

\begin{remark}[Geometric meaning of $\sigma$, $\alpha$, and $\varphi$]\label{rem:geometric-meaning}
The three permutations have the following geometric interpretation.
\begin{itemize}
\item \emph{$\sigma$: rotation around a vertex.} Think of a vertex $v$ as a disk with $d$ marked points on its boundary, arranged in the cyclic order given by the surface orientation. The half-edges at $v$ are these marked boundary points, and $\sigma$ advances one step counterclockwise: $\sigma(h)$ is the next half-edge at the same vertex after $h$. Each $\sigma$-orbit is one vertex, of length equal to its valence.
\item \emph{$\alpha$: pairing across an edge.} An edge connects two half-edges, one at each end; $\alpha$ swaps them. Since every half-edge lies on exactly one edge, $\alpha$ is fixed-point-free.
\item \emph{$\varphi = \sigma\alpha$: traversal around a face.} Starting at a half-edge $h$, apply $\alpha$ to cross to the other end of its edge, then $\sigma$ to rotate counterclockwise to the next half-edge there. The $\varphi$-orbit of $h$ traces the entire boundary of the face containing $h$.
\end{itemize}
With vertices, edges, faces realized as $\sigma$-, $\alpha$-, $\varphi$-orbits, the Euler characteristic is $\chi = V - E + F$ with $E = |H|/2$.
\end{remark}

\begin{definition}\label{def:iso}
Ribbon graphs $G = (H,\sigma,\alpha)$ and $G' = (H',\sigma',\alpha')$ are \emph{isomorphic} if there is a bijection $\psi : H \to H'$ with $\psi\sigma\psi^{-1} = \sigma'$ and $\psi\alpha\psi^{-1} = \alpha'$. An \emph{automorphism} is an isomorphism $G \to G$; the automorphism group is $\Aut(G)$.
\end{definition}

To each ribbon graph $G$ one associates an oriented \emph{topological realization} $\Sigma(G)$: replace each degree-$d$ vertex by a disk with $d$ marked boundary points in the prescribed cyclic order, replace each edge by a rectangle gluing matching boundary points, and cap off every remaining boundary component by a disk. Its Euler characteristic is $\chi(\Sigma(G)) = V - E + F$, and its genus is $g(G) := 1 - \chi(\Sigma(G))/2$.

\subsection{The Petrie permutation}\label{ssec:petrie}

For a $4$-regular ribbon graph there is a second natural permutation, recording how strands pass \emph{straight through} each vertex.

\begin{definition}\label{def:petrie}
Let $G = (H,\sigma,\alpha)$ be a $4$-regular ribbon graph. The \emph{Petrie permutation} of $G$ is
\[
\tau \;:=\; \sigma^2\alpha .
\]
Its orbits are the \emph{zigzag paths} (or \emph{Petrie polygons}) of $G$.
\end{definition}

At a $4$-valent vertex with cyclic half-edge order $(a,b,c,d)$ one has $\sigma^2 = (a\,c)(b\,d)$, which pairs each half-edge with the one diametrically opposite. Thus $\tau = \sigma^2\alpha$ encodes the walk that crosses an edge (via $\alpha$) and then exits a vertex through the opposite half-edge (via $\sigma^2$), i.e. proceeds ``straight on'' rather than turning. This is precisely the valence-$4$ instance of the classical Petrie operator, whose orbits in a general map are the zigzag paths that turn alternately left and right~\cite{LandoZvonkin2004}; for $4$-regular maps ``alternately left and right'' specializes to ``straight through,'' and the operator becomes $\sigma^2\alpha$. 
We record two structural facts for later use.

\begin{lemma}[Parity of $\tau$]\label{lem:tau-parity}
Let $G$ be a $4$-regular ribbon graph with $V$ vertices. Then $\sgn(\tau) = +1$, i.e.\ $\tau \in A_{4V}$.
\end{lemma}

\begin{proof}
$\sigma^2$ and $\alpha$ are fixed-point-free involutions on the $4V$ half-edges, each a product of $2V$ transpositions, so both are even and so is $\tau = \sigma^2\alpha$.
\end{proof}

\begin{lemma}[Cycle-pairing structure of $\tau$]\label{lem:tau-doubled}
The cycle type of $\tau = \sigma^2\alpha$ is a \emph{doubled partition}: every length occurs with even multiplicity. In particular the number of $\tau$-cycles is even.
\end{lemma}

\begin{proof}
Consider the graph on the half-edges whose edges are the transpositions of $\sigma^2$ and of $\alpha$. Both are fixed-point-free involutions, so every half-edge lies on exactly one edge of each kind, and each component of this graph is a cycle $v_0, v_1, \dots, v_{2\ell-1}$ alternating between the two kinds, say with $\alpha(v_{2i}) = v_{2i+1}$ and $\sigma^2(v_{2i+1}) = v_{2i+2}$, indices modulo $2\ell$. Then
\[
\tau(v_{2i}) = \sigma^2(v_{2i+1}) = v_{2i+2}, \qquad \tau(v_{2i+1}) = \sigma^2(v_{2i}) = v_{2i-1},
\]
so on each component $\tau$ acts as two $\ell$-cycles, one on the even-indexed and one on the odd-indexed half-edges. Hence every cycle length of $\tau$ occurs an even number of times.
\end{proof}

\subsection{A primer on the representation theory of $S_N$}\label{ssec:char-bkgd}

We provide a brief recollection of the facts about characters of $S_N$ that will be needed; more detailed exposition can be found in~\cite{MondalThesis}, and some standard references for the material are~\cite{Serre1977,Stanley1999}. Readers familiar with the Frobenius formalism may skip to Section~\ref{sec:encoding}.

\subsubsection*{Characters}

A representation of a finite group is a homomorphism $\rho : G \to GL(V)$, $V$ a finite-dimensional complex vector space, and its character is the class function $\chi(\pi) := \mathrm{tr}\,\rho(\pi)$, which determines $\rho$ up to isomorphism. The irreducible representations of $S_N$ are indexed by partitions $\lambda \vdash N$; we write $\chi^\lambda$ for the character of $\rho^\lambda$, $\chi^\lambda(K)$ for its common value on a conjugacy class $K$, and $\chi^\lambda(1) = \dim\rho^\lambda$. All characters of $S_N$ are real. Three irreducibles are distinguished: the \emph{trivial} $\rho^{(N)}$, with $\chi^{(N)} \equiv 1$; the \emph{sign} $\rho^{(1^N)}$, with $\chi^{(1^N)} = \sgn$; and the \emph{standard} representation $\rho^{(N-1,1)}$ on $\Vstd = \{x \in \mathbb{C}^N : \sum x_i = 0\}$, the complement of the trivial summand in the permutation representation $\mathbb{C}^N$, with $\chi^{(N-1,1)}(\pi) = \#\mathrm{fix}(\pi) - 1$.

\subsubsection*{The Frobenius formula}

Irreducible characters are orthogonal, $\sum_{\pi \in S_N} \chi^\lambda(\pi)\chi^\mu(\pi) = N!\,\delta_{\lambda\mu}$, so class functions expand in the basis $\{\chi^\lambda\}$ and convolution of class functions diagonalizes in it, exactly as Fourier series diagonalize convolution. Applied to the normalized indicator functions of two conjugacy classes, this yields the \textbf{Frobenius character formula}: for conjugacy classes $K_1, K_2 \subseteq S_N$ and any $h \in S_N$,
\begin{equation}\label{eq:frobenius}
\#\{(g_1,g_2) \in K_1 \times K_2 : g_1 g_2 = h\}
= \frac{|K_1|\,|K_2|}{N!} \sum_{\lambda \vdash N} \frac{\chi^\lambda(K_1)\,\chi^\lambda(K_2)\,\chi^\lambda(h)}{\chi^\lambda(1)}
\end{equation}
\cite[Ch.~7]{Serre1977}. Taking $K_1 = C_\sigma$, $K_2 = C_\alpha$ and $h \in C_N$ gives the labelled count $A_g$ defined at the end of Section~\ref{sec:encoding} after summing over $C_N$. As written this is a sum of $p(N) \sim e^{C\sqrt N}$ terms, but for an $N$-cycle almost all of them vanish.

\subsubsection*{Hooks and the Murnaghan--Nakayama collapse}

The \emph{hook partitions} are $\lambda = (N-k, 1^k)$, $k = 0,\dots,N-1$: a row of length $N-k$ atop a column of length $k+1$. There are $N$ of them, $k = 0$ being the trivial and $k = N-1$ the sign representation.

\begin{lemma}\label{lem:MN-Ncycle}
For an $N$-cycle $c \in S_N$, $\chi^\lambda(c) = 0$ unless $\lambda$ is a hook, in which case $\chi^{(N-k,1^k)}(c) = (-1)^k$.
\end{lemma}

\begin{proof}
By the Murnaghan--Nakayama rule, $\chi^\lambda(c)$ is a signed count of ways to remove a border strip of length $N$ from $\lambda$ leaving the empty diagram. A connected strip of length $N$ with no $2\times2$ block is an L-shape, so $\lambda$ must be a hook $(N-k,1^k)$, with sign $(-1)^k$~\cite[Cor.~7.17.5]{Stanley1999}.
\end{proof}

Thus with $h$ an $N$-cycle the sum in \eqref{eq:frobenius} collapses from $p(N)$ partitions to $N$ hooks. The hook characters are explicit: $\chi^{(N-k,1^k)}$ is the character of the exterior power $\bigwedge^k\Vstd$, so
\begin{equation}\label{eq:hook-ext}
\dim\chi^{(N-k,1^k)} = \binom{N-1}{k},
\end{equation}
and since the eigenvalues of $\pi$ on $\mathbb{C}^N$ are the $\ell$th roots of unity for each $\ell$-cycle of $\pi$ (one eigenvalue $+1$ being discarded on passing to $\Vstd$), the hook character values have the generating function
\begin{equation}\label{eq:hook-genfn}
\sum_{k=0}^{N-1} \chi^{(N-k,1^k)}(\pi)\,x^k = \frac{1}{1+x}\prod_{\ell \geq 1}\bigl(1-(-x)^\ell\bigr)^{m_\ell(\pi)},
\end{equation}
where $m_\ell(\pi)$ is the number of $\ell$-cycles of $\pi$~\cite[Ex.~7.72]{Stanley1999}. Evaluating \eqref{eq:hook-genfn} at the cycle types of $\sigma$ and $\alpha$ gives all $N$ hook character values at once (Lemma~\ref{lem:hookvals}). In the resulting hook sum for $p_g = A_g/(|C_\sigma||C_\alpha|)$, the trivial and sign terms contribute exactly $2/N$, reproducing the parity heuristic of Lemma~\ref{lem:parity}, and the remaining $N-2$ hooks contribute $O(N^{-2})$ (See Section~\ref{sec:probability}).

\section{From curve systems to permutations}\label{sec:encoding}

This section builds the correspondence between minimally intersecting filling systems and pairs of permutations and proves Theorem~\ref{thmA}.

\subsection{The curve-system--ribbon graph dictionary}\label{ssec:dictionary}

We now explain precisely how a generic filling curve system $\Gamma$ on $\Sg$ determines a ribbon graph; see Figure~\ref{fig:dictionary}.

\begin{proposition}\label{prop:dictionary}
A generic filling curve system $\Gamma$ on $\Sg$ canonically determines a $4$-regular oriented ribbon graph $G_\Gamma$ on $\Sg$. The complementary faces of $G_\Gamma$ are in bijection with the components of $\Sg \setminus \Gamma$; in particular $G_\Gamma$ is unicellular if and only if $\Gamma$ has exactly one complementary disk.
\end{proposition}

\begin{proof}
At each transverse double point (crossing) $p$ of $\Gamma$, four arc-germs (directed arc-ends) meet. Label these half-edges $a, b, c, d$ in the counterclockwise cyclic order determined by the orientation of $\Sg$, and define $\sigma$ locally by the $4$-cycle $(a\,b\,c\,d)$: this encodes the crossing as a $4$-valent vertex.

Each arc of $\Gamma$ connecting two crossings (or returning to the same one) contributes one half-edge at each of its two ends; the involution $\alpha$ pairs these. Since each arc-end belongs to exactly one arc, $\alpha$ is a fixed-point-free involution.

The resulting $G_\Gamma = (H, \sigma, \alpha)$ is $4$-regular, and its topological realization is homeomorphic to $\Sg$ (a tubular neighborhood of $\Gamma$ together with the complementary disks). The face orbits of $\varphi = \sigma\alpha$ trace the boundaries of the complementary regions of $\Gamma$, giving a bijection between faces and complementary disks.
\end{proof}

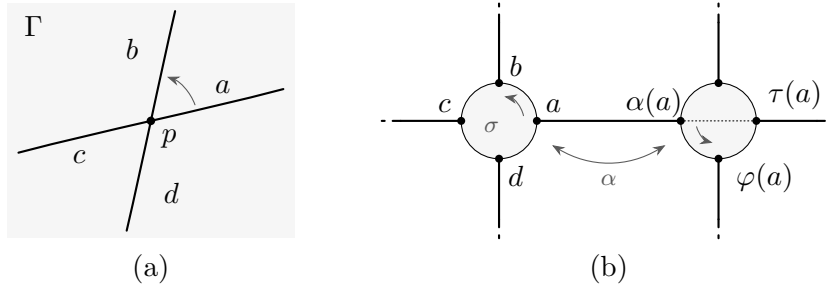
\begin{figure}[ht]
\centering
\begin{tikzpicture}[>=Stealth, line cap=round]
\begin{scope}
  \fill[gray!7] (-1.9,-1.55) rectangle (1.9,1.55);
  \draw[thick] (-1.75,-0.42) .. controls (-0.6,-0.12) and (0.6,0.12) .. (1.75,0.42);
  \draw[thick] (-0.32,-1.45) .. controls (-0.1,-0.5) and (0.1,0.5) .. (0.32,1.45);
  \fill (0,0) circle (1.6pt);
  \node at (0.24,-0.24) {$p$};
  \node at (0.95,0.45)  {$a$};
  \node at (-0.25,0.95) {$b$};
  \node at (-0.95,-0.47){$c$};
  \node at (0.27,-0.95) {$d$};
  \draw[->, gray!70!black] (20:0.62) arc (20:70:0.62);
  \node at (-1.55,1.25) {$\Gamma$};
  \node at (0,-1.95) {(a)};
\end{scope}
\begin{scope}[xshift=4.6cm]
  \draw[thick] (0.5,0) -- (2.4,0);
  \draw[thick] (0,0.5) -- (0,1.3);   \draw[thick, dashed] (0,1.3) -- (0,1.55);
  \draw[thick] (-0.5,0) -- (-1.3,0); \draw[thick, dashed] (-1.3,0) -- (-1.55,0);
  \draw[thick] (0,-0.5) -- (0,-1.3); \draw[thick, dashed] (0,-1.3) -- (0,-1.55);
  \draw[thick] (2.9,0.5) -- (2.9,1.3);   \draw[thick, dashed] (2.9,1.3) -- (2.9,1.55);
  \draw[thick] (3.4,0) -- (4.2,0);       \draw[thick, dashed] (4.2,0) -- (4.45,0);
  \draw[thick] (2.9,-0.5) -- (2.9,-1.3); \draw[thick, dashed] (2.9,-1.3) -- (2.9,-1.55);
  \foreach \x in {0,2.9} {\filldraw[fill=gray!7, draw=black] (\x,0) circle (0.5);
     \foreach \t in {0,90,180,270} \fill (\x,0) ++(\t:0.5) circle (1.5pt);}
  \node at (0.72,0.2) {$a$};
  \node at (0.22,0.72) {$b$};
  \node at (-0.72,0.2) {$c$};
  \node at (0.22,-0.72) {$d$};
  \draw[->, gray!70!black] (12:0.33) arc (12:78:0.33);
  \node[gray!70!black] at (-0.1,-0.08) {\footnotesize$\sigma$};
  \draw[<->, gray!70!black] (0.7,-0.32) .. controls (1.2,-0.62) and (1.7,-0.62) .. (2.2,-0.32);
  \node[gray!70!black] at (1.45,-0.78) {\footnotesize$\alpha$};
  \node at (2.05,0.24) {$\alpha(a)$};
  \node[anchor=west] at (3.0,-0.75) {$\varphi(a)$};
  \node[anchor=south west] at (3.42,0.02) {$\tau(a)$};
  \draw[->, gray!70!black] (2.9,0) ++(195:0.3) arc (195:255:0.3);
  \draw[densely dotted] (2.4,0) -- (3.4,0);
  \node at (1.45,-1.95) {(b)};
\end{scope}
\end{tikzpicture}
\caption{(a) A double point $p$ of $\Gamma$; the four arc-germs at $p$ are labelled $a,b,c,d$ in the counterclockwise order given by the orientation of $\Sg$. (b) The corresponding vertex of $G_\Gamma$, a disk with four marked boundary points in the same cyclic order, and the vertex at the other end of the arc through $a$. The permutation $\sigma$ rotates counterclockwise around a vertex and $\alpha$ exchanges the two ends of an edge; thus $\varphi(a) = \sigma\alpha(a)$ turns at the next vertex, while $\tau(a) = \sigma^2\alpha(a)$ passes straight through it (dotted).}
\label{fig:dictionary}
\end{figure}

\begin{remark}[Components are encoded by the Petrie permutation]\label{rem:single-vs-multi}
Whereas one step of $\varphi$ crosses an arc and turns at the next crossing (Remark~\ref{rem:geometric-meaning}), one step of $\tau = \sigma^2\alpha$ crosses an arc and passes straight through to the opposite half-edge (\S\ref{ssec:petrie}), i.e.\ follows a geometric strand of $\Gamma$ (Figure~\ref{fig:dictionary}(b)). Each oriented strand is therefore one $\tau$-orbit, reversing its orientation gives a second $\tau$-orbit of the same length (the pairing of Lemma~\ref{lem:tau-doubled}), and the number of unoriented components is $c(\Gamma) = \tfrac12\#\{\tau\text{-orbits}\}$. This number is determined by $G_\Gamma$ but is not visible from the cycle types of $\sigma$ and $\alpha$ alone; it is the basis of the single-curve criterion in Theorem~\ref{thm:dictionary}(ii).
\end{remark}

\begin{proposition}\label{prop:euler}
If $\Gamma$ is a generic filling curve system on $\Sg$ with $k$ double points and $b$ complementary disks, then $k = 2g - 2 + b$. In particular $\iota(\Gamma) \geq 2g - 1$, with equality if and only if $\Gamma$ is minimally intersecting.
\end{proposition}

\begin{proof}
The ribbon graph $G_\Gamma$ has $V = k$ vertices, $E = 2k$ edges (each crossing contributes four half-edges, each edge using two), and $F = b$ faces. Euler's identity $V - E + F = 2 - 2g$ gives $k - 2k + b = 2 - 2g$, so $k = 2g - 2 + b$, and $b \geq 1$ is equivalent to filling.
\end{proof}

\subsection{The three-level dictionary}\label{ssec:threelevel}

The proof of Theorem~\ref{thm:main-asymptotic} passes through three equivalent descriptions of a minimally intersecting filling curve system. We record the full correspondence here; in the right-hand column, $N = 8g-4$ and all permutations act on the half-edge set $[N]$.

\begin{center}
\renewcommand{\arraystretch}{1.35}
\begin{tabular}{@{}>{\raggedright\arraybackslash}p{0.30\textwidth} >{\raggedright\arraybackslash}p{0.31\textwidth} >{\raggedright\arraybackslash}p{0.31\textwidth}@{}}
\toprule
\textbf{Topology of $\Gamma$ on $\Sg$} & \textbf{Ribbon graph $G_\Gamma$} & \textbf{Permutations on $[N]$} \\
\midrule
genus-$g$ surface $\Sg$ & realization $\Sigma(G)$ of genus $g$ & genus read from $V-E+F$ \\
crossing (double point) & $4$-valent vertex & $4$-cycle of $\sigma$; \ type $4^{2g-1}$ \\
arc between crossings & edge & transposition of $\alpha$; \ type $2^{4g-2}$ \\
complementary disk & face & cycle of $\varphi=\sigma\alpha$ \\
\emph{one} disk (minimal) & unicellular & $\varphi=\sigma\alpha\in C_N$ \\
oriented strand & directed zigzag path & cycle of $\tau=\sigma^2\alpha$ \\
curve component & unoriented zigzag (pair) & pair of $\tau$-cycles \\
\emph{one} component & single zigzag, length $4g-2$ & $\tau=\sigma^2\alpha\in C_{(N/2)^2}$ \\
$\MCG(\Sg)$-orbit of $\Gamma$ & isomorphism class of $G_\Gamma$ & simultaneous-conjugacy class of $(\sigma,\alpha)$ \\
\bottomrule
\end{tabular}
\end{center}

\noindent The two rows---unicellularity ($\varphi\in C_N$) and single-component ($\tau\in C_{(N/2)^2}$)---are the two defining conditions of problems (M) and (S) respectively. The last row is justified in the next subsection.

\subsection{Mapping class group orbits via ribbon graphs}\label{ssec:bijection}

A labelling of the half-edges by $[N] = \{1,\dots,N\}$ is an arbitrary choice of coordinates: a different labelling gives a different pair $(\sigma',\alpha')$ representing the same ribbon graph. Any bijection $\psi \in S_N$ relabels $(\sigma,\alpha)$ to $(\psi\sigma\psi^{-1}, \psi\alpha\psi^{-1})$, and two labelled pairs represent isomorphic ribbon graphs if and only if they are related by such a $\psi$ (Definition~\ref{def:iso}). Under the dictionary, the action of $\MCG(\Sg)$ on filling curve systems corresponds to this conjugation action. Indeed, an orientation-preserving homeomorphism $h$ taking $\Gamma_1$ to $\Gamma_2$ sends crossings to crossings and arcs to arcs, inducing $\psi : H_1 \to H_2$ that carries the cyclic order at each crossing to the cyclic order at its image, so $\psi\sigma_1\psi^{-1} = \sigma_2$, and carries each arc to an arc, so $\psi\alpha_1\psi^{-1} = \alpha_2$. Conversely, such a $\psi$ is exactly an isomorphism $G_{\Gamma_1} \to G_{\Gamma_2}$, which extends over the uniquely determined complementary disks to a homeomorphism $\Sg \to \Sg$. Hence two labelled pairs represent the same $\MCG$-orbit iff they are simultaneously conjugate in $S_N$. This is the topological half of the classification; Theorem~\ref{thm:dictionary} restates it in permutation language and adds the single-curve refinement.

\begin{proposition}\label{prop:bijection}
The assignment $\Gamma \mapsto G_\Gamma$ induces a bijection
\[
\left\{
\begin{tabular}{@{}c@{}}
minimally intersecting filling \\
curve systems on $\Sg$
\end{tabular}
\right\} \Big/ \MCG(\Sg)
\;\xrightarrow{\ \cong\ }\;
\left\{
\begin{tabular}{@{}c@{}}
unicellular oriented $4$-regular \\
ribbon graphs with $2g-1$ vertices
\end{tabular}
\right\} \Big/ {\cong}.
\]
\end{proposition}

\begin{proof}
\emph{Injectivity} is the argument above: an isomorphism $\psi : G_{\Gamma_1} \to G_{\Gamma_2}$ is realized by a homeomorphism of regular neighborhoods carrying $\Gamma_1$ to $\Gamma_2$, which extends across the unique complementary disk on each side. Since the cyclic ordering in $\sigma$ fixes a local surface orientation, this homeomorphism is orientation-preserving; accordingly $\mathcal{N}_g$ counts orientation-preserving $\MCG(\Sg)$-orbits.

\emph{Surjectivity.} A unicellular $4$-regular ribbon graph $G$ with $V = 2g-1$, $E = 4g-2$, $F = 1$ has $\Sigma(G) \cong \Sg$. The embedded $4$-valent graph determines a curve system by \emph{smoothing} each vertex according to its cyclic order: at a vertex with cyclic order $(a,b,c,d)$, declare the strands $\{a,c\}$ and $\{b,d\}$ to be the two passes of $\Gamma$ through the crossing. The result is a generic filling system with one complementary disk and $G_\Gamma = G$.

\emph{Finiteness.} The pairs $(\sigma,\alpha)$ on $[N]$ with prescribed cycle structures form a finite set, and isomorphism classes are orbits of a finite group action.
\end{proof}

\subsection{The permutation encoding}\label{ssec:encoding}

A unicellular $4$-regular ribbon graph with $V = 2g - 1$ vertices has $E = 4g - 2$ edges and $N := 4V = 8g - 4$ half-edges. Label the half-edges by $[N]$ and identify a labelled ribbon graph with the pair $(\sigma,\alpha) \in S_N \times S_N$ where
\begin{itemize}
\item $\alpha \in C_\alpha :=$ the class of $4g-2$ disjoint transpositions (type $2^{4g-2}$);
\item $\sigma \in C_\sigma :=$ the class of $2g-1$ disjoint $4$-cycles (type $4^{2g-1}$).
\end{itemize}
The face permutation is $\varphi = \sigma\alpha$ and the Petrie permutation is $\tau = \sigma^2\alpha$; unicellularity is the condition $\varphi \in C_N$ (Definition~\ref{def:ribbon}).

\begin{theorem}\label{thm:dictionary}\label{thm:dictionary-restate}
Let $g \geq 2$ and $N = 8g-4$. The dictionary $\Gamma \mapsto (\sigma,\alpha)$ induces:
\begin{enumerate}
\item[\textup{(i)}] a bijection between $\MCG(\Sg)$-orbits of minimally intersecting filling curve \emph{systems} on $\Sg$ and simultaneous-conjugacy classes of pairs $(\sigma,\alpha) \in C_\sigma \times C_\alpha$ with $\varphi = \sigma\alpha \in C_N$;
\item[\textup{(ii)}] under this bijection, the \emph{single} filling curves are exactly the pairs additionally satisfying $\tau = \sigma^2\alpha \in C_{(N/2)^2}$, equivalently those for which $G^\ast = (\sigma^2,\alpha)$ is connected.
\end{enumerate}
\end{theorem}

\begin{proof}
\emph{(i)} By Proposition~\ref{prop:bijection}, $\MCG$-orbits of systems biject with isomorphism classes of unicellular $4$-regular ribbon graphs on $2g-1$ vertices. In the labelling, a vertex is a $4$-cycle of $\sigma$, so the $2g-1$ vertices give $\sigma \in C_\sigma$; an edge is a transposition of $\alpha$, giving $\alpha \in C_\alpha$; and the single face (unicellularity) is $\varphi = \sigma\alpha \in C_N$ by Lemma~\ref{lem:unicellular}. Isomorphism of labelled ribbon graphs is simultaneous conjugacy (\S\ref{ssec:bijection}).

\emph{(ii)} By Remark~\ref{rem:single-vs-multi}, the unoriented components of $\Gamma$ are the orientation-reversal pairs of $\tau$-orbits, so $c(\Gamma) = \tfrac12\#\{\tau\text{-orbits}\}$. Both $\sigma^2$ and $\alpha$ are fixed-point-free involutions, so by Lemma~\ref{lem:tau-doubled} every $\tau$-orbit is paired with another of equal length; hence $c(\Gamma) = 1$ iff $\tau$ has exactly two orbits, necessarily each of length $N/2 = 4g-2$, i.e. $\tau \in C_{(N/2)^2}$. Finally $G^\ast = (\sigma^2,\alpha)$ has vertex permutation $\sigma^2$ of type $2^{4g-2}$, hence is $2$-regular with face permutation $\tau$; each connected component contributes two $\tau$-orbits, so $\tau$ has two cycles iff $G^\ast$ is connected.
\end{proof}

\begin{lemma}\label{lem:conjclass}
$\displaystyle |C_\alpha| = \frac{N!}{2^{4g-2}\,(4g-2)!}, \qquad |C_\sigma| = \frac{N!}{4^{2g-1}\,(2g-1)!}.$
\end{lemma}

\begin{proof}
For $C_\alpha$: arrange the $N$ elements in a sequence ($N!$ ways) and form the first $4g-2$ adjacent pairs as transpositions; this overcounts by the $2^{4g-2}$ orderings within the pairs and the $(4g-2)!$ orderings of the pairs, giving $|C_\alpha| = N!/(2^{4g-2}(4g-2)!)$. For $C_\sigma$: similarly arrange $N$ elements ($N!$ ways) into $2g-1$ consecutive blocks of four read as $4$-cycles; this overcounts by the $4^{2g-1}$ cyclic rotations within the blocks and the $(2g-1)!$ orderings of the blocks, giving $|C_\sigma| = N!/(4^{2g-1}(2g-1)!)$.
\end{proof}

\begin{lemma}\label{lem:unicellular}
$(\sigma,\alpha)$ is unicellular iff $\varphi = \sigma\alpha$ is an $N$-cycle.
\end{lemma}

\begin{proof}
By Remark~\ref{rem:geometric-meaning} the faces are the orbits of $\varphi$, so one face means $\varphi \in C_N$.
\end{proof}

\paragraph{Strategy for the count.} By Theorem~\ref{thm:dictionary}(i) we must count $S_N$-conjugacy classes of pairs $(\sigma,\alpha) \in C_\sigma \times C_\alpha$ with $\sigma\alpha \in C_N$. The strategy is as follows:
\begin{enumerate}
\item \textbf{Count labelled pairs:} $A_g := \#\{(\sigma,\alpha) \in C_\sigma \times C_\alpha : \sigma\alpha \in C_N\}$ (Section~\ref{sec:probability}).
\item \textbf{Convert to conjugacy classes:} each class contributes $N!/|\Aut(G)|$ labelled pairs, so $A_g = \sum_{[G]} N!/|\Aut(G)|$.
\item \textbf{Correct for automorphisms:} since the classes with non-trivial automorphism group are negligible (Lemma~\ref{lem:asymmetry}), $A_g/N! \sim \mathcal{N}_g$ (Section~\ref{sec:asymptotic}).
\end{enumerate}

\section{The probability of unicellularity}\label{sec:probability}

We define the labelled unicellular pair count and probability as follows
\[
A_g := \#\{(\sigma,\alpha)\in C_\sigma\times C_\alpha : \sigma\alpha\in C_N\}, \qquad
p_g := \frac{A_g}{|C_\sigma||C_\alpha|} = \Pr_{\sigma,\alpha}[\sigma\alpha\in C_N].
\]

\begin{theorem}\label{thm:probability}
$\displaystyle p_g = \frac{2}{N} + O(N^{-2})$ as $g \to \infty$, with $N = 8g-4$.
\end{theorem}

The proof occupies the rest of this section. A parity argument explains the leading term $2/N$ (Lemma~\ref{lem:parity}); the Frobenius formula then expresses $p_g$ as a sum over the $N$ hook characters, whose values on $C_\sigma$ and $C_\alpha$ are computed in Lemma~\ref{lem:hookvals}; the trivial and sign hooks reproduce $2/N$ exactly, and the remaining $N-2$ hooks contribute $O(N^{-2})$ (Lemma~\ref{lem:mid-bound}).

\begin{lemma}[Parity]\label{lem:parity}
For every $(\sigma,\alpha) \in C_\sigma \times C_\alpha$: $\sgn(\sigma) = -1$, $\sgn(\alpha) = +1$, $\sgn(\sigma\alpha) = -1$. Every $N$-cycle also has sign $-1$.
\end{lemma}

\begin{proof}
A $4$-cycle has sign $-1$ and $\sigma$ is a product of $2g-1$ (odd) of them, so $\sgn(\sigma)=-1$; a transposition has sign $-1$ and $\alpha$ is a product of $4g-2$ (even) of them, so $\sgn(\alpha)=+1$; hence $\sgn(\sigma\alpha)=-1$. Finally $\sgn(c) = (-1)^{N-1} = -1$ since $N$ is even.
\end{proof}

The $(N-1)!$ many $N$-cycles all lie in the odd coset $S_N\setminus A_N$ of size $N!/2$, which by Lemma~\ref{lem:parity} also contains $\sigma\alpha$. If $\sigma\alpha$ were uniform on this coset, the probability of being an $N$-cycle would be $(N-1)!/(N!/2) = 2/N$, matching Theorem~\ref{thm:probability}. The character computation turns this heuristic into a theorem and bounds the error.

\begin{lemma}[Hook character values]\label{lem:hookvals}
For $0 \leq k \leq N-1$,
\begin{equation}\label{eq:chi-alpha-val}
\chi^{(N-k,1^k)}(\alpha) =
\begin{cases}
(-1)^{k/2}\binom{4g-3}{k/2}, & k\text{ even},\\
(-1)^{(k+1)/2}\binom{4g-3}{(k-1)/2}, & k\text{ odd},
\end{cases}
\end{equation}
and, writing $k = 4m+r$ with $r\in\{0,1,2,3\}$,
\begin{equation}\label{eq:chi-sigma-val}
\chi^{(N-k,1^k)}(\sigma) =
\begin{cases}
(-1)^m\binom{2g-2}{m}, & r\in\{0,2\},\\
(-1)^{m+1}\binom{2g-2}{m}, & r\in\{1,3\}.
\end{cases}
\end{equation}
\end{lemma}

\begin{proof}
From \eqref{eq:hook-genfn}, on $C_\alpha$ (type $2^{4g-2}$) the generating function of the hook character values is
\[
\sum_k \chi^{(N-k,1^k)}(\alpha)x^k = \frac{(1-x^2)^{4g-2}}{1+x} = (1-x)(1-x^2)^{4g-3},
\]
and on $C_\sigma$ (type $4^{2g-1}$), using $1-x^4=(1-x)(1+x)(1+x^2)$,
\[
\sum_k \chi^{(N-k,1^k)}(\sigma)x^k = \frac{(1-x^4)^{2g-1}}{1+x} = (1-x)(1+x^2)(1-x^4)^{2g-2}.
\]
Reading off the coefficient of $x^k$ in each gives \eqref{eq:chi-alpha-val} and \eqref{eq:chi-sigma-val}.
\end{proof}

The error term in Theorem~\ref{thm:probability} comes from the hooks with $1 \leq k \leq N-2$. Set
\[
T_k := \frac{(-1)^k\,\chi^{(N-k,1^k)}(\sigma)\,\chi^{(N-k,1^k)}(\alpha)}{\binom{N-1}{k}}, \qquad 0 \leq k \leq N-1.
\]
By Lemma~\ref{lem:hookvals}, writing $k = 4m+r$ and $s := \lfloor r/2\rfloor$ (so that $\lfloor k/2 \rfloor = 2m+s$),
\begin{equation}\label{eq:Tk-bound-num}
|T_k| \leq \frac{\binom{2g-2}{m}\binom{4g-3}{2m+s}}{\binom{8g-5}{4m+r}}.
\end{equation}

\begin{lemma}[Triple Chu--Vandermonde bound]\label{lem:vandermonde}
For $k = 4m+r$ with $r \in \{0,1,2,3\}$ and $0 \leq m \leq 2g-2$,
\[
\binom{8g-5}{4m+r} \geq \binom{2g-2}{m}\binom{4g-3}{2m+s}\binom{2g}{m+r-s}.
\]
Moreover, for $k \in \{1,\dots,N-2\}$ the index $n_k := m+r-s$ satisfies $1 \leq n_k \leq 2g-1$.
\end{lemma}

\begin{proof}
By Chu--Vandermonde, selecting the $j=m$ term,
\[
\binom{8g-5}{4m+r} = \sum_{j}\binom{2g-2}{j}\binom{6g-3}{4m+r-j} \geq \binom{2g-2}{m}\binom{6g-3}{3m+r},
\]
and again with $6g-3=(4g-3)+2g$, $3m+r=(2m+s)+(m+r-s)$,
\[
\binom{6g-3}{3m+r} \geq \binom{4g-3}{2m+s}\binom{2g}{m+r-s}.
\]
Since $r-s\in\{0,1,1,2\}$ all arguments are non-negative. At $k=1$, $n_k=1$; at $k=N-2$, $n_k=2g-1$; intermediate values lie between, so $\binom{2g}{n_k}\geq 2g>0$.
\end{proof}

\begin{lemma}\label{lem:mid-bound}
$\displaystyle \sum_{k=1}^{N-2}|T_k| = O(g^{-1})$.
\end{lemma}

\begin{proof}
By \eqref{eq:Tk-bound-num} and Lemma~\ref{lem:vandermonde}, $|T_k| \leq 1/\binom{2g}{n_k}$ with $n_k\in\{1,\dots,2g-1\}$, and each $n$ is hit by at most four values of $k$. Thus
\[
\sum_{k=1}^{N-2}|T_k| \leq 4\sum_{n=1}^{2g-1}\frac{1}{\binom{2g}{n}} = \frac{8}{2g} + \sum_{n=2}^{2g-2}\frac{4}{\binom{2g}{n}} = O(g^{-1}),
\]
since interior terms are $O(g^{-2})$.
\end{proof}

\begin{proof}[Proof of Theorem~\ref{thm:probability}]
Specializing \eqref{eq:frobenius} to $K_1 = C_\sigma$, $K_2 = C_\alpha$, $h\in C_N$, summing over $C_N$, and dividing by $|C_\sigma||C_\alpha|$ gives (using $|C_N|/N! = 1/N$)
\begin{equation}\label{eq:p-frobenius}
p_g = \frac{1}{N}\sum_{\lambda \vdash N} \frac{\chi^\lambda(C_\sigma)\,\chi^\lambda(C_\alpha)\,\chi^\lambda(C_N)}{\chi^\lambda(1)}.
\end{equation}
By Lemma~\ref{lem:MN-Ncycle} only hooks survive, with $\chi^{(N-k,1^k)}(C_N) = (-1)^k$, and $\chi^{(N-k,1^k)}(1) = \binom{N-1}{k}$ by \eqref{eq:hook-ext}; hence
\begin{equation}\label{eq:hook-sum}
p_g = \frac{1}{N}\sum_{k=0}^{N-1} \frac{(-1)^k\,\chi^{(N-k,1^k)}(C_\sigma)\,\chi^{(N-k,1^k)}(C_\alpha)}{\binom{N-1}{k}} = \frac{1}{N}\sum_{k=0}^{N-1} T_k.
\end{equation}
The $k=0$ term is $T_0 = 1$ ($\chi^{(N)}\equiv1$). For $k = N-1$ (sign), \eqref{eq:chi-sigma-val} with $m = 2g-2$, $r=3$ gives $\chi^{(1^N)}(\sigma) = -1$, and \eqref{eq:chi-alpha-val} with $k$ odd gives $\chi^{(1^N)}(\alpha) = +1$, matching the signs in Lemma~\ref{lem:parity}; since $(-1)^{N-1} = -1$, also $T_{N-1} = 1$. These two terms contribute exactly $2/N$ to \eqref{eq:hook-sum}, the heuristic value, and by Lemma~\ref{lem:mid-bound} the remaining terms contribute
\[
\frac{1}{N}\sum_{k=1}^{N-2} T_k = N^{-1}\,O(g^{-1}) = O(N^{-2}). \qedhere
\]
\end{proof}

\begin{remark}
The exact $A_g$ has closed forms via Goupil--Schaeffer~\cite{GoupilSchaeffer1998} or Chapuy's recursion~\cite{Chapuy2011}; both give the same $p_g = 2/N + O(N^{-2})$. Lemma~\ref{lem:mid-bound} establishes this directly.
\end{remark}

\begin{corollary}\label{cor:Ag-asymp}
$\displaystyle A_g \sim \frac{2}{N}\cdot\frac{N!}{2^{4g-2}(4g-2)!}\cdot\frac{N!}{4^{2g-1}(2g-1)!}$ as $g\to\infty$.
\end{corollary}

\section{Asymptotic enumeration}\label{sec:asymptotic}

By Theorem~\ref{thm:dictionary}(i), $\mathcal{N}_g$ is the number of orbits of $S_N$ acting by simultaneous conjugation on the labelled unicellular pairs, and the stabilizer of $(\sigma,\alpha)$ is $\Aut(G)$. By orbit-stabilizer,
\begin{equation}\label{eq:orbit-aut}
A_g = \sum_{[G]} \frac{N!}{|\Aut(G)|},
\end{equation}
the sum running over isomorphism classes of unicellular $4$-regular ribbon graphs with $2g-1$ vertices. Since every term is at most $N!$,
\begin{equation}\label{eq:easy-lower}
\frac{A_g}{N!} \;=\; \sum_{[G]} \frac{1}{|\Aut(G)|} \;\leq\; \mathcal{N}_g,
\end{equation}
and the defect is carried entirely by the classes with non-trivial automorphism group:
\begin{equation}\label{eq:defect}
\mathcal{N}_g - \frac{A_g}{N!} \;=\; \sum_{[G]:\,\Aut(G)\neq\{e\}}\Bigl(1 - \frac{1}{|\Aut(G)|}\Bigr) \;\leq\; \mathcal{U}_g^{\mathrm{sym}},
\end{equation}
where $\mathcal{U}_g^{\mathrm{sym}}$ denotes the number of such classes. Theorem~\ref{thmB} therefore follows from Corollary~\ref{cor:Ag-asymp} once $\mathcal{U}_g^{\mathrm{sym}}$ is shown to be negligible against $A_g/N!$. This is Lemma~\ref{lem:asymmetry} in \S\ref{ssec:rooting}; \S\ref{ssec:proof-main} assembles the proof.

\subsection{Automorphisms of unicellular maps}\label{ssec:rooting}

\begin{lemma}[Automorphism structure]\label{lem:aut-cyclic}
Let $G = (\sigma,\alpha)$ be unicellular with $\varphi = \sigma\alpha \in C_N$. Then $\Aut(G)$ is a subgroup of $\langle\varphi\rangle \cong \mathbb{Z}/N$, hence cyclic of some order $d \mid N$ and generated by $\varphi^{N/d}$; every non-identity automorphism acts on the half-edge set $H = [N]$ without fixed points.
\end{lemma}

\begin{proof}
An automorphism commutes with $\sigma$ and $\alpha$, hence with $\varphi$, and the centralizer of an $N$-cycle in $S_N$ is $\langle\varphi\rangle$. Non-identity powers of an $N$-cycle have no fixed points.
\end{proof}

Let $\mathcal{R}_g$ be the number of \emph{rooted} unicellular $4$-regular ribbon graphs with $2g-1$ vertices, i.e.\ with a distinguished half-edge.

\begin{lemma}\label{lem:rooted}
$\displaystyle \mathcal{R}_g = \frac{A_g}{(N-1)!} = N\sum_{[G]}\frac{1}{|\Aut(G)|}$, and each class $[G]$ accounts for exactly $N/|\Aut(G)|$ rooted graphs.
\end{lemma}

\begin{proof}
Root at the half-edge $1$. The relabelling action of $S_{N-1}$ (the permutations fixing $1$) on rooted labelled pairs is free: a relabelling stabilizing $(\sigma,\alpha)$ lies in $\Aut(G) \leq \langle\varphi\rangle$ and fixes $1$ only if it is the identity (Lemma~\ref{lem:aut-cyclic}). Hence $\mathcal{R}_g = A_g/(N-1)!$, and the second equality is \eqref{eq:orbit-aut}. Rooted graphs with underlying graph $G$ correspond to $\Aut(G)$-orbits on $H$, and this action is free by Lemma~\ref{lem:aut-cyclic}, so there are $N/|\Aut(G)|$ of them.
\end{proof}

The proof that symmetric maps are rare has three ingredients: the quotient of a map by an automorphism of prime order (Lemma~\ref{lem:quotient}), and two general estimates in $S_n$ (Lemmas~\ref{lem:factor-bound} and~\ref{lem:class-size}). For a conjugacy class $T$ of $S_n$ with $m_\ell$ cycles of length $\ell$ we write $v(T) = \sum_\ell m_\ell$ for its number of cycles, and likewise $v(\pi)$ for a permutation $\pi$.

\begin{lemma}[Quotient by a prime-order automorphism]\label{lem:quotient}
Let $p \mid N$ be prime, $\bar N := N/p$, $\varphi_0 := (1\,2\cdots N)$ and $c_0 := \varphi_0^{\bar N}$, whose orbits $B_i := \{i,\, i+\bar N,\, i+2\bar N,\dots\}$, $1 \leq i \leq \bar N$, we call \emph{blocks}. Let $\mathcal{M}_p$ be the set of $\sigma \in C_\sigma$ that commute with $c_0$ and satisfy $\alpha := \sigma^{-1}\varphi_0 \in C_\alpha$. Then:
\begin{enumerate}
\item[\textup{(i)}] each $\sigma \in \mathcal{M}_p$ permutes the blocks, inducing $\bar\sigma \in S_{\bar N}$; likewise $\alpha$ induces an involution $\bar\alpha$, and $\bar\sigma\bar\alpha = (1\,2\cdots\bar N)$;
\item[\textup{(ii)}] every cycle of $\bar\sigma$ has length $4$, or length $2$ when $p = 2$;
\item[\textup{(iii)}] each $\bar\sigma$ arises from at most $p^{\,3(2g-1)/p} = e^{O(N)}$ elements of $\mathcal{M}_p$.
\end{enumerate}
\end{lemma}

\begin{proof}
(i) Both $\sigma$ and $\alpha = \sigma^{-1}\varphi_0$ commute with $c_0$, so they lie in $Z := \Cent_{S_N}(c_0) \cong \mathbb{Z}_p \wr S_{\bar N}$ and permute the blocks; let $\pi : Z \to S_{\bar N}$ be the induced action. Since $\varphi_0(B_i) = B_{i+1}$, $\pi(\varphi_0) = (1\,2\cdots\bar N)$; as $\pi$ is a homomorphism, $\bar\sigma\bar\alpha = \pi(\sigma\alpha) = \pi(\varphi_0)$, and $\alpha^2 = 1$ gives $\bar\alpha^2 = 1$.

(ii), (iii) Identify each block with $\mathbb{Z}_p$ so that $c_0$ acts by $j \mapsto j+1$. An element of $Z$ over $\bar\sigma$ is then determined by the shifts $\theta_1,\dots,\theta_{\bar N} \in \mathbb{Z}_p$ it induces on the blocks. On a $\bar\sigma$-cycle $(B_{i_1}\cdots B_{i_\ell})$ one has $\sigma^\ell|_{B_{i_1}} = c_0^{\,m}|_{B_{i_1}}$ with $m = \theta_{i_1}+\dots+\theta_{i_\ell}$, so on these $\ell p$ half-edges $\sigma$ has cycle type $\ell^{\,p}$ if $m = 0$ and $(\ell p)^1$ if $m \neq 0$, since $c_0^{\,m}$ is then a $p$-cycle. As $\sigma$ has only $4$-cycles, every $\bar\sigma$-cycle has $\ell = 4$ and $m = 0$, or $p = 2$, $\ell = 2$ and $m \neq 0$. This proves (ii), and it prescribes, for each $\bar\sigma$-cycle, the value of a linear form in that cycle's coordinates; distinct cycles involve disjoint coordinates, so at most $p^{\,\bar N - v(\bar\sigma)}$ elements of $Z$ over $\bar\sigma$ lie in $C_\sigma$. Finally $\bar N - v(\bar\sigma) \leq 3(2g-1)/p$: for $p$ odd $\bar\sigma$ has type $4^{(2g-1)/p}$, and for $p = 2$ type $4^k2^f$ with $2k+f = 2g-1$, so that $\bar N - v(\bar\sigma) = 3k+f \leq \tfrac32(2g-1)$.
\end{proof}

\begin{lemma}[Factorizations of a long cycle]\label{lem:factor-bound}
Let $T, U$ be conjugacy classes of $S_n$ and let $\gamma$ be an $n$-cycle. Then
\[
\#\{x \in T : x^{-1}\gamma \in U\} \;\leq\; \frac{|T|\,|U|}{n!}\,2^{\,n-1}.
\]
\end{lemma}

\begin{proof}
The left side counts pairs $(x,y) \in T \times U$ with $xy = \gamma$. By the Frobenius formula~\eqref{eq:frobenius} and Lemma~\ref{lem:MN-Ncycle} (with $N$ replaced by $n$) it equals $\frac{|T||U|}{n!}\sum_{j=0}^{n-1}(-1)^j\chi^{(n-j,1^j)}(T)\,\chi^{(n-j,1^j)}(U)\big/\binom{n-1}{j}$, and since $|\chi^{(n-j,1^j)}| \leq \dim\chi^{(n-j,1^j)} = \binom{n-1}{j}$ by \eqref{eq:hook-ext}, each term is at most $\binom{n-1}{j}$; summing gives $2^{n-1}$.
\end{proof}

\begin{lemma}[Classes with short cycles are small]\label{lem:class-size}
Let $L \geq 1$ be fixed and let $T$ be a conjugacy class of $S_n$ all of whose cycles have length at most $L$. Then
\[
|T| \;\leq\; \exp\!\Bigl(\bigl(1-\tfrac1L\bigr)\,n\log n + O_L(n)\Bigr).
\]
\end{lemma}

\begin{proof}
$|T| = n!\big/\prod_\ell \ell^{m_\ell}\,m_\ell! \leq n!\big/\prod_\ell m_\ell!$. The $v = v(T) \geq n/L$ cycles are spread over at most $L$ lengths, so by $\log m! \geq m\log m - m$ and convexity of $x\log x$,
\[
\sum_\ell \log m_\ell! \;\geq\; \sum_\ell m_\ell\log m_\ell - v \;\geq\; v\log\tfrac{v}{L} - v \;\geq\; \tfrac{n}{L}\log n - O_L(n).
\]
With $\log n! \leq n\log n$ this gives the claim.
\end{proof}

\begin{lemma}[Symmetric maps are rare]\label{lem:asymmetry}
For every $K > 0$,
\[
\mathcal{U}_g^{\mathrm{sym}} \;\leq\; \exp\!\Bigl(-\tfrac18\,N\log N + O(N)\Bigr)\cdot\frac{A_g}{N!} \;=\; O_K\bigl(\mathcal{N}_g\,g^{-K}\bigr) \qquad (g\to\infty).
\]
\end{lemma}

\begin{proof}
\emph{Reduction to a fixed cycle and a prime.}
By Lemma~\ref{lem:aut-cyclic}, a symmetric class has $|\Aut(G)| = d$ with $1 < d \mid N$; it contributes $N!/d \geq (N-1)!$ labelled pairs to \eqref{eq:orbit-aut}, and its automorphism group contains $\varphi^{N/p}$ for every prime $p \mid d$. A union bound over primes, and conjugation invariance to fix $\varphi = \varphi_0$, give
\[
(N-1)!\,\mathcal{U}_g^{\mathrm{sym}} \;\leq\; \sum_{p \mid N}\#\bigl\{(\sigma,\alpha) \in C_\sigma\times C_\alpha : \sigma\alpha \in C_N,\ (\sigma\alpha)^{N/p} \in \Aut(\sigma,\alpha)\bigr\} \;=\; (N-1)!\sum_{p \mid N} M_p,
\]
where $M_p := |\mathcal{M}_p|$ in the notation of Lemma~\ref{lem:quotient}: indeed $\alpha = \sigma^{-1}\varphi_0$ is determined by $\sigma$, and $c_0$, which commutes with $\varphi_0$, commutes with $\alpha$ iff it commutes with $\sigma$.

\emph{Counting $\mathcal{M}_p$.}
By Lemma~\ref{lem:quotient}, each $\sigma \in \mathcal{M}_p$ determines a quotient $\bar\sigma \in S_{\bar N}$ whose cycles have length $2$ or $4$ and for which $\bar\sigma^{-1}(1\,2\cdots\bar N)$ is an involution, and at most $e^{O(N)}$ elements of $\mathcal{M}_p$ share a given $\bar\sigma$. Sorting the possible $\bar\sigma$ by the conjugacy classes $T \ni \bar\sigma$ and $U \ni \bar\sigma^{-1}(1\,2\cdots\bar N)$---there are $O(\bar N^2)$ such pairs $(T,U)$---Lemmas~\ref{lem:factor-bound} and~\ref{lem:class-size} (with $L = 4$ for $T$ and $L = 2$ for $U$) and $\log\bar N! = \bar N\log\bar N + O(\bar N)$ give
\[
\#\{\bar\sigma\} \;\leq\; \sum_{(T,U)} \frac{|T|\,|U|}{\bar N!}\,2^{\bar N-1} \;\leq\; \exp\!\Bigl(\bigl(\tfrac34+\tfrac12-1\bigr)\bar N\log\bar N + O(\bar N)\Bigr),
\]
hence $M_p \leq \exp\bigl(\tfrac{\bar N}{4}\log\bar N + O(\bar N)\bigr) \leq \exp\bigl(\tfrac{N}{4p}\log N + O(N)\bigr)$.

\emph{Comparison with $A_g/N!$.}
By Corollary~\ref{cor:Ag-asymp} and Stirling, $\log(A_g/N!) = \tfrac14 N\log N + O(N)$, so
\[
\log\frac{M_p}{A_g/N!} \;\leq\; -\frac N4\Bigl(1-\frac1p\Bigr)\log N + O(N) \;\leq\; -\frac N8\log N + O(N),
\]
the bound being weakest at $p = 2$; summing over the $O(\log N)$ primes dividing $N$ gives the displayed inequality. Finally $\tfrac18 N\log N = g\log g\,(1+o(1))$, so the prefactor is $g^{-g(1+o(1))} = O_K(g^{-K})$ for every $K$, and $A_g/N! \leq \mathcal{N}_g$ by \eqref{eq:easy-lower}.
\end{proof}

\begin{remark}[Odd $p$]\label{rem:odd-p-quotient}
For $p$ odd, the count of quotients can be replaced by an identification: $c_0$ fixes no vertex and no edge setwise (on a vertex it would act as an element of $\langle\sigma|_V\rangle$ of order dividing $\gcd(p,4) = 1$, on an edge with order dividing $\gcd(p,2) = 1$, hence trivially in both cases), so the quotient $(\bar\sigma,\bar\alpha)$ of Lemma~\ref{lem:quotient} is itself a unicellular $4$-regular ribbon graph, of genus $\bar g$ with $8\bar g - 4 = \bar N$, and $\#\{\bar\sigma\} \leq \mathcal{R}_{\bar g}$. For $p = 2$ this fails---$2g-1$ is odd, so $c_0$ fixes a vertex and acts on it as the half-turn $\sigma^2|_V$, which becomes $2$-valent in the quotient---and $p = 2$ is also where the comparison with $A_g/N!$ is weakest, so the general count cannot be dispensed with.
\end{remark}

\begin{corollary}\label{cor:N-asymp-Ag}
For every $K > 0$, as $g\to\infty$,
\[
\mathcal{N}_g = \frac{A_g}{N!}\bigl(1 + O_K(g^{-K})\bigr), \qquad \mathcal{R}_g = N\,\mathcal{N}_g\bigl(1 + O_K(g^{-K})\bigr);
\]
in particular $\mathcal{N}_g \sim A_g/N!$.
\end{corollary}

\begin{proof}
The first statement is \eqref{eq:defect} together with Lemma~\ref{lem:asymmetry}; the second follows from $\mathcal{R}_g = N\,A_g/N!$ (Lemma~\ref{lem:rooted}).
\end{proof}

\subsection{Proof of the main asymptotic}\label{ssec:proof-main}

\begin{theorem}\label{thm:main-asymptotic}
As $g \to \infty$, $\displaystyle \mathcal{N}_g \sim \frac{16^g\,(2g)!}{64\,\pi\sqrt{2}\,g^3}$.
\end{theorem}

\begin{proof}
By Corollaries~\ref{cor:Ag-asymp} and~\ref{cor:N-asymp-Ag},
\begin{equation}\label{eq:N-explicit}
\mathcal{N}_g \sim \frac{A_g}{N!} \sim \frac{2}{N}\cdot\frac{N!}{2^{8g-4}\,(4g-2)!\,(2g-1)!},
\end{equation}
using $2^{4g-2}\cdot4^{2g-1} = 2^{8g-4}$. By the trinomial identity $N!/[(4g-2)!\,(2g-1)!] = \binom{N}{4g-2}\binom{4g-2}{2g-1}(2g-1)!$, and Stirling's approximations
\[
\binom{8g-4}{4g-2}\sim\frac{2^{8g-4}}{2\sqrt{\pi g}}, \qquad \binom{4g-2}{2g-1}\sim\frac{2^{4g-2}}{\sqrt{2\pi g}}.
\]
Substituting into \eqref{eq:N-explicit} with $1/N\sim 1/(8g)$ and $(2g-1)! = (2g)!/(2g)$,
\[
\mathcal{N}_g \sim \frac{2}{8g}\cdot\frac{1}{2^{8g-4}}\cdot\frac{2^{8g-4}}{2\sqrt{\pi g}}\cdot\frac{2^{4g-2}}{\sqrt{2\pi g}}\cdot\frac{(2g)!}{2g}
= \frac{2^{4g-6}\,(2g)!}{\pi\sqrt{2}\,g^3}
= \frac{16^g\,(2g)!}{64\,\pi\sqrt{2}\,g^3}. \qedhere
\]
\end{proof}

\section{Computations in low genus}\label{sec:lowgenus}

For small $g$, the counts can be computed explicitly. The hook sum~\eqref{eq:hook-sum}, with the values \eqref{eq:chi-alpha-val}--\eqref{eq:chi-sigma-val}, gives $p_g$ exactly as a finite sum of rationals, and with it $A_g = p_g\,|C_\sigma||C_\alpha|$, the rooted count $\mathcal{R}_g = A_g/(N-1)!$ and $A_g/N! = \sum_{[G]} 1/|\Aut(G)|$, for any $g$. The individual classes---and hence $\mathcal{N}_g$, the automorphism orders and the component counts---are obtained by enumeration: fixing $\varphi = (1\,2\cdots N)$, the rooted graphs are the fixed-point-free involutions $\alpha$ with $\varphi\alpha \in C_\sigma$ (Lemma~\ref{lem:rooted}); a backtracking search lists them (Appendix~\ref{sec:algorithm}), conjugation by $\langle\varphi\rangle$ groups them into classes, with $|\Aut(G)| = N/|\mathrm{orbit}|$, and the cycle type of $\tau = \sigma^2\alpha$ gives the number of components of each class.

\paragraph{Exact values.}\label{ssec:exact-values}
The rooted counts from the character sum are
\begin{center}
\renewcommand{\arraystretch}{1.2}
\begin{tabular}{@{}rrl@{}}
\toprule
$g$ & $N$ & $\mathcal{R}_g$ \\
\midrule
$2$  & $12$ & $45$ \\
$3$  & $20$ & $9{,}450$ \\
$4$  & $28$ & $4{,}729{,}725$ \\
$5$  & $36$ & $4{,}341{,}887{,}550$ \\
$6$  & $44$ & $6{,}352{,}181{,}485{,}650$ \\
$7$  & $52$ & $13{,}566{,}444{,}744{,}352{,}500$ \\
$8$  & $60$ & $39{,}834{,}473{,}380{,}605{,}028{,}125$ \\
$9$  & $68$ & $153{,}946{,}961{,}458{,}244{,}898{,}693{,}750$ \\
$10$ & $76$ & $757{,}572{,}997{,}336{,}023{,}146{,}471{,}943{,}750$ \\
\bottomrule
\end{tabular}
\end{center}
The enumeration has to list all $\mathcal{R}_g$ rooted graphs, so it is feasible only while $\mathcal{R}_g$ is small: it gives the complete breakdown of the classes for $g \leq 3$, and for $g = 4$, where $\mathcal{R}_4 \approx 4.7\times10^6$, the number of classes alone, obtained with Burnside's lemma (Appendix~\ref{sec:algorithm}). Thus
\[
\mathcal{N}_2 = 6, \qquad \mathcal{N}_3 = 510, \qquad \mathcal{N}_4 = 169{,}772 .
\]
For $g = 2$ and $g = 3$ the enumeration gives the full breakdown of the classes by number of components:
\begin{center}
\renewcommand{\arraystretch}{1.2}
\begin{tabular}{@{}lcccccc@{}}
\toprule
number of components $c$ & $1$ & $2$ & $3$ & $4$ & $5$ & $6$ \\
\midrule
classes at $g = 2$ \ ($\mathcal{N}_2 = 6$)   & $1$   & $2$   & $2$   & $1$  &      &     \\
classes at $g = 3$ \ ($\mathcal{N}_3 = 510$) & $108$ & $197$ & $134$ & $56$ & $13$ & $2$ \\
\bottomrule
\end{tabular}
\end{center}
In particular $\mathcal{N}_2^{(1)} = 1$ and $\mathcal{N}_3^{(1)} = 108$, so $\mathcal{N}_2^{(1)}/\mathcal{N}_2 = 1/6$ and $\mathcal{N}_3^{(1)}/\mathcal{N}_3 \approx 0.212$; the labelled probabilities of \S\ref{ssec:single-data}, which weight each class by $1/|\Aut(G)|$, are $P_2 = 4/15$ and $P_3 = 8/35$. The enumeration also returns the automorphism orders: $1,1,2,2,2,4$ at $g = 2$, and multiplicities $1^{436}\,2^{72}\,4^{2}$ at $g = 3$. The resulting values $\sum_{[G]} 1/|\Aut(G)| = \tfrac{15}{4}$ and $\tfrac{945}{2}$ agree with $A_g/N!$ from the character sum ($p_2 = 32/231$, $p_3 = 4096/46189$), a check on both computations.

\begin{remark}\label{rem:g2-unique}
In particular for $g=2$, there is a \emph{unique} minimally intersecting filling single curve, up to the action of $\MCG(S_2)$.
\end{remark}

\paragraph{Comparison with the asymptotic.}\label{ssec:convergence}
Writing $\rho_g$ for the ratio of the exact $\mathcal{R}_g$ to the leading term $N\cdot 16^g(2g)!/(64\pi\sqrt2\,g^3)$ given by Theorem~\ref{thmB} and Corollary~\ref{cor:N-asymp-Ag},
\[
\begin{array}{c|ccccccccc}
g & 2 & 3 & 4 & 5 & 6 & 7 & 8 & 9 & 10\\\hline
\rho_g & 1.388 & 1.230 & 1.163 & 1.127 & 1.103 & 1.087 & 1.076 & 1.067 & 1.060
\end{array}
\]
The ratios decrease steadily toward $1$, consistent with a positive correction of relative size $O(g^{-1})$; at $g = 10$ the leading term still falls short of $\mathcal{R}_g$ by about six percent. For the unrooted count the ratios $\mathcal{N}_g/(16^g(2g)!/(64\pi\sqrt2\,g^3))$ are $2.221$, $1.328$, $1.169$ at $g = 2,3,4$; the automorphism correction, super-polynomially small by Lemma~\ref{lem:asymmetry}, is already small at $g = 4$, where $\mathcal{N}_4$ exceeds $A_4/N!$ by about $0.5\%$. Likewise $N p_g/2 = 0.831, 0.887, 0.915, 0.932, 0.943, 0.951$ for $g = 2,\dots,7$ approaches $1$ from below, in agreement with the negative $O(N^{-2})$ correction in Theorem~\ref{thm:probability}.

\section{The single-curve subproblem}\label{sec:singlecurves}

This section addresses problem (S). By Theorem~\ref{thm:dictionary}(ii) the single curves are the pairs $(\sigma,\alpha)$ with $\tau = \sigma^2\alpha \in C_{(N/2)^2}$, so what is needed is the distribution of the Petrie permutation $\tau$ of a uniformly random pair under the face condition $\varphi \in C_N$. We first determine this distribution without the face condition (Theorem~\ref{thm:two-matchings}), then show that the face condition enters only through a single explicit one-face probability $\rho_U$ (Theorem~\ref{thm:gelfand-reduction}). The single-curve asymptotic is thereby reduced to one limit, $N\rho_U \to 2$, which we conjecture and test numerically in \S\ref{ssec:single-data}; the algorithm behind the numerical evidence is described in Appendix~\ref{app:sampler}.

\subsection{Reduction to a single one-face probability}\label{ssec:single-results}

\begin{corollary}\label{cor:single-upper}
For every $g \geq 2$, $\mathcal{N}_g^{(1)} \leq \mathcal{N}_g$, so
$\mathcal{N}_g^{(1)} \leq \frac{16^g(2g)!}{64\pi\sqrt2\,g^3}(1+o(1))$.
\end{corollary}

\begin{proof}
A minimally intersecting filling single curve is a minimally intersecting filling system, and distinct $\MCG$-orbits of single curves are distinct $\MCG$-orbits of systems, so $\mathcal{N}_g^{(1)} \leq \mathcal{N}_g$; Theorem~\ref{thm:main-asymptotic} gives the bound.
\end{proof}

This improves the upper bound of Parsad and Sanki~\cite{ParsadSanki2022} for single filling curves; for the related count of minimally intersecting filling pairs, see~\cite{AougabHuang2015,AougabMenascoNieland}. The conjectured truth (Conjecture~\ref{conj:single-curve}) is smaller by a factor of order $g^{-1/2}$. By Lemmas~\ref{lem:tau-parity}--\ref{lem:tau-doubled}, $\tau$ always lies in the set $\mathcal{D}_N \subset A_N$ of even permutations with doubled cycle type; the next two statements describe its law there.

\begin{lemma}\label{lem:squaring-surjective}
The map $C_\sigma \to C_\alpha$, $\sigma \mapsto \sigma^2$, is surjective with uniform fibers: every $\beta \in C_\alpha$ has exactly $(4g-2)!/(2g-1)!$ preimages. Hence $\sigma^2$ is uniform on $C_\alpha$ when $\sigma$ is uniform on $C_\sigma$.
\end{lemma}

\begin{proof}
A $4$-cycle squares to two disjoint transpositions, $(a\,b\,c\,d)^2 = (a\,c)(b\,d)$, and conversely $(a\,c)(b\,d)$ has exactly two $4$-cycle square roots, $(a\,b\,c\,d)$ and $(a\,d\,c\,b)$. A preimage of $\beta$ is therefore obtained by grouping the $4g-2$ transpositions of $\beta$ into $2g-1$ pairs and choosing a square root for each pair, which can be done in
\[
\frac{(4g-2)!}{2^{2g-1}\,(2g-1)!}\cdot 2^{2g-1} \;=\; \frac{(4g-2)!}{(2g-1)!}
\]
ways, independently of $\beta$. Since all fibers have the same size, the uniform measure on $C_\sigma$ maps to the uniform measure on $C_\alpha$.
\end{proof}

Thus, for $(\sigma,\alpha)$ uniform on $C_\sigma\times C_\alpha$, the pair $(\sigma^2,\alpha)$ consists of two independent uniform fixed-point-free involutions, and $\tau = \sigma^2\alpha$ is distributed as the product of two random perfect matchings. We need the probability of the single-cycle type.

\begin{theorem}[Unconditional Petrie law]\label{thm:two-matchings}
Let $\beta_1,\beta_2$ be independent uniform fixed-point-free involutions of $[N]$, $N=2n$. Then
\[
q_g := \Pr\bigl[\beta_1\beta_2 \in C_{(N/2)^2}\bigr]
     = \prod_{j=1}^{n-1}\frac{2j}{2j+1}
     = \frac{2^{N-1}\,(N/2-1)!\,(N/2)!}{N!}
     \;\sim\; \sqrt{\frac{\pi}{2N}} \qquad (g\to\infty),
\]
so that $\sqrt{N}\,q_g \to \sqrt{\pi/2}=1.25331\ldots$. Consequently, for $(\sigma,\alpha)$ uniform on $C_\sigma\times C_\alpha$, $\Pr[\tau\in C_{(N/2)^2}] = q_g$ exactly.
\end{theorem}

\begin{proof}
By Lemma~\ref{lem:tau-doubled} and its proof, each component of the union graph $\beta_1 \cup \beta_2$ with $2\ell$ vertices contributes two $\ell$-cycles to $\beta_1\beta_2$. Hence $\beta_1\beta_2 \in C_{(N/2)^2}$ if and only if $\beta_1\cup\beta_2$ is a single cycle through all $2n$ points.

Fix $\beta_2$ and a point $v_0$, and reveal $\beta_1$ along the alternating walk from $v_0$:
\[
b_1 = \beta_2(v_0), \quad c_1 = \beta_1(b_1), \quad b_2 = \beta_2(c_1), \quad c_2 = \beta_1(b_2), \quad \dots
\]
The walk closes up, into a cycle of length $2j$, exactly when $c_j = v_0$. Since $\beta_1$ is uniform, its restriction to the points not yet matched is again a uniform matching. So if the walk has not closed before step $j$, then $c_j$ is uniform among the $2n-2j+1$ points other than $b_j$ that $\beta_1$ has not yet matched, and $v_0$ is one of them. Hence
\[
\Pr[\text{the walk does not close at step } j \mid \text{it has not closed before}] = \frac{2n-2j}{2n-2j+1}, \qquad 1 \leq j \leq n-1,
\]
while at step $n$ only $v_0$ remains and the walk closes. The union is a single cycle iff the walk first closes at step $n$, so, substituting $k = n-j$,
\[
q_g = \prod_{k=1}^{n-1}\frac{2k}{2k+1} = \frac{2^{n-1}(n-1)!}{(2n-1)!!} = \frac{2^{2n-1}\,(n-1)!\,n!}{(2n)!} = \frac{4^n}{2n}\binom{2n}{n}^{-1},
\]
which is the stated closed form with $2n = N$. Stirling's formula $\binom{2n}{n} \sim 4^n/\sqrt{\pi n}$ then gives
\[
q_g \sim \frac{1}{2}\sqrt{\frac{\pi}{n}} = \sqrt{\frac{\pi}{2N}}.
\]
Finally, for $(\sigma,\alpha)$ uniform on $C_\sigma\times C_\alpha$, the pair $(\sigma^2,\alpha)$ is a pair of independent uniform fixed-point-free involutions by Lemma~\ref{lem:squaring-surjective}, so $\Pr[\tau \in C_{(N/2)^2}] = q_g$.
\end{proof}
Note that $q_g = \Theta(N^{-1/2})$, whereas a uniform distribution on $\mathcal{D}_N$ would give $2|C_{(N/2)^2}|/N! = (N/2)^{-2}$: the product of two random matchings has $\sim\log N$ cycles, concentrated on a few long ones, so the single-cycle type is far more likely than a uniform model predicts.

\paragraph{The face condition.}
Define the conditional single-curve probability and the ratio
\[
P_g := \Pr\bigl[\tau\in C_{(N/2)^2}\,\bigm|\,\varphi\in C_N\bigr],
\qquad
R_g := \frac{P_g}{q_g}.
\]
The labelled probability $P_g$ governs the counts up to a super-polynomially small error.

\begin{lemma}\label{lem:N1-vs-Pg}
For every $K > 0$, $\ \mathcal{N}_g^{(1)}/\mathcal{N}_g = P_g + O_K(g^{-K})$ as $g \to \infty$.
\end{lemma}

\begin{proof}
Let $A_g^{(1)}$ be the number of labelled pairs $(\sigma,\alpha)\in C_\sigma\times C_\alpha$ with $\varphi\in C_N$ and $\tau\in C_{(N/2)^2}$, so that $P_g = A_g^{(1)}/A_g$. Restricting \eqref{eq:orbit-aut} to the single-curve classes gives
\[
\frac{A_g^{(1)}}{N!} \;=\; \sum_{[G]\ \text{single curve}} \frac{1}{|\Aut(G)|}.
\]
Each term is at most $1$, with equality when $\Aut(G)$ is trivial, so
\[
0 \;\leq\; \mathcal{N}_g^{(1)} - \frac{A_g^{(1)}}{N!} \;\leq\; \mathcal{U}_g^{\mathrm{sym}} \;=\; O_K\bigl(\mathcal{N}_g\,g^{-K}\bigr)
\]
by Lemma~\ref{lem:asymmetry}. Dividing by $\mathcal{N}_g$, and using $N!\,\mathcal{N}_g = A_g\bigl(1+O_K(g^{-K})\bigr)$ from Corollary~\ref{cor:N-asymp-Ag}, we obtain
\[
\frac{\mathcal{N}_g^{(1)}}{\mathcal{N}_g} \;=\; \frac{A_g^{(1)}}{A_g}\bigl(1+O_K(g^{-K})\bigr) + O_K(g^{-K}) \;=\; P_g + O_K(g^{-K}),
\]
the last step because $P_g \leq 1$.
\end{proof}

By simultaneous-conjugation invariance (\S\ref{ssec:bijection}) we may fix $\alpha=\alpha_0\in C_\alpha$ and let $\sigma$ be uniform on $C_\sigma$. For a matching $\beta\in C_\alpha$ set
\[
\rho(\beta) \;:=\; \Pr_{\sigma\in C_\sigma}\bigl[\varphi=\sigma\alpha_0\in C_N \,\bigm|\, \sigma^2=\beta\bigr],
\]
well defined because the squaring fibers are uniform (Lemma~\ref{lem:squaring-surjective}). Let $H_n := \Cent_{S_N}(\alpha_0)\cong \mathbb{Z}_2\wr S_n$ be the hyperoctahedral group, $n=N/2$.

\begin{lemma}[$H_n$-invariance of $\rho$]\label{lem:rho-invariant}
$\rho(h\beta h^{-1})=\rho(\beta)$ for every $h\in H_n$. Consequently $\rho$ is constant on each $H_n$-conjugation orbit of matchings.
\end{lemma}

\begin{proof}
The fiber size $(4g-2)!/(2g-1)!$ is independent of $\beta$, so it suffices to match one-face counts. Conjugation $\sigma\mapsto h\sigma h^{-1}$ is a bijection $\{\sigma:\sigma^2=\beta\}\to\{\sigma:\sigma^2=h\beta h^{-1}\}$, and as $h\alpha_0=\alpha_0 h$,
\[
(h\sigma h^{-1})\,\alpha_0 = h\,\sigma\,(h^{-1}\alpha_0) = h\,\sigma\,(\alpha_0 h^{-1}) = h\,(\sigma\alpha_0)\,h^{-1},
\]
which has the same cycle type as $\sigma\alpha_0$; in particular it is an $N$-cycle iff $\sigma\alpha_0$ is. Thus the two one-face counts coincide and $\rho(h\beta h^{-1})=\rho(\beta)$.
\end{proof}

The $H_n$-orbits on matchings are the \emph{coset types} of the Gelfand pair $(S_N,H_n)$, indexed by partitions $\lambda\vdash n$, where $\lambda$ records the half-lengths of the connected components of $\beta\cup\alpha_0$ (equivalently, $\beta\alpha_0$ has cycle type $\lambda\cup\lambda$)~\cite[Ch.~VII]{Macdonald1995}. The orbit of coset type $(n)$ consists of those $\beta$ with $\beta\cup\alpha_0$ connected, i.e. $\beta\alpha_0\in C_{(N/2)^2}$; write $\rho_U$ for the (constant) value of $\rho$ there.

\begin{theorem}[Gelfand-pair reduction]\label{thm:gelfand-reduction}
For every $g\ge2$,
\[
\Pr\bigl[\tau\in C_{(N/2)^2},\ \varphi\in C_N\bigr] = \rho_U\,q_g,
\qquad\text{hence}\qquad
P_g = q_g\,\frac{\rho_U}{p_g},
\qquad R_g = \frac{\rho_U}{p_g}.
\]
In particular $P_g$ is an exactly computable rational for every $g$, and
\[
P_g \sim \sqrt{\frac{\pi}{2N}}
\quad\Longleftrightarrow\quad R_g \to 1
\quad\Longleftrightarrow\quad \rho_U \sim p_g \sim \frac2N
\quad\Longleftrightarrow\quad N\rho_U \to 2 .
\]
\end{theorem}

\begin{proof}
Fix $\alpha = \alpha_0$ and write $\beta := \sigma^2$, which is uniform on $C_\alpha$ by Lemma~\ref{lem:squaring-surjective}. Let $\mathcal{O}$ be the $H_n$-orbit of coset type $(n)$. The event $\tau \in C_{(N/2)^2}$ reads $\beta\alpha_0 \in C_{(N/2)^2}$, that is, $\beta \in \mathcal{O}$, so it depends on $\beta$ alone. Conditioning on $\beta$, and using that $\rho = \rho_U$ on $\mathcal{O}$ (Lemma~\ref{lem:rho-invariant}),
\[
\Pr\bigl[\tau\in C_{(N/2)^2},\ \varphi\in C_N\bigr]
\;=\; \sum_{\beta \in \mathcal{O}} \Pr[\sigma^2 = \beta]\,\rho(\beta)
\;=\; \rho_U\,\Pr[\sigma^2 \in \mathcal{O}]
\;=\; \rho_U\,q_g,
\]
the last equality by Theorem~\ref{thm:two-matchings}. Dividing by $\Pr[\varphi\in C_N] = p_g$ gives $P_g = q_g\rho_U/p_g$, and $R_g = \rho_U/p_g$ by definition.

The three factors are rational: $q_g$ by its closed form, $p_g$ as the finite sum \eqref{eq:hook-sum}, and $\rho_U$ as the proportion of one-face maps among the $(4g-2)!/(2g-1)!$ lifts of any fixed $\beta_* \in \mathcal{O}$. The equivalences follow from $q_g \sim \sqrt{\pi/(2N)}$ (Theorem~\ref{thm:two-matchings}) and $p_g \sim 2/N$ (Theorem~\ref{thm:probability}).
\end{proof}

\subsection{Conjecture and evidence}\label{ssec:single-data}

By Lemma~\ref{lem:N1-vs-Pg} and Theorem~\ref{thm:gelfand-reduction}, the asymptotics of $\mathcal{N}_g^{(1)}/\mathcal{N}_g$ are those of $P_g = q_g\,\rho_U/p_g$, in which $q_g$ and $p_g$ are known; everything hinges on the single limit $N\rho_U\to2$, equivalently $R_g\to1$. The latter has a direct interpretation: $\{\varphi\in C_N\}$ depends on $\sigma$ while $\{\tau\in C_{(N/2)^2}\}$ depends only on $\sigma^2$, and $R_g\to1$ says that these events become asymptotically independent.

\begin{conjecture}[Petrie independence]\label{conj:decoupling}
As $g\to\infty$ the face condition is asymptotically independent of the Petrie cycle type: for every fixed doubled cycle type $\theta\in\mathcal D_N$,
\[
\Pr[\tau\in C_\theta \mid \varphi\in C_N] = \Pr[\tau\in C_\theta]\,(1+o(1)).
\]
Equivalently $R_g\to1$, i.e. $\rho_U \sim p_g$.
\end{conjecture}

\begin{conjecture}\label{conj:single-curve}
As $g\to\infty$,
\[
P_g \sim q_g \sim \sqrt{\frac{\pi}{2N}},
\qquad
\frac{\mathcal{N}_g^{(1)}}{\mathcal{N}_g}\sim\frac14\sqrt{\frac{\pi}{g}},
\qquad
\mathcal{N}_g^{(1)} \sim \frac{16^g(2g)!}{256\,\sqrt{2\pi}\,g^{7/2}}.
\]
\end{conjecture}

Conjecture~\ref{conj:single-curve} is the case $\theta=(N/2)^2$ of Conjecture~\ref{conj:decoupling} combined with Lemma~\ref{lem:N1-vs-Pg} and Theorem~\ref{thm:main-asymptotic}: from $P_g\sim q_g\sim\sqrt{\pi/(2N)}$ and $1/N\sim 1/(8g)$,
\[
\mathcal{N}_g^{(1)} \sim \mathcal{N}_g\sqrt{\frac{\pi}{2N}}
\sim \frac{16^g(2g)!}{64\pi\sqrt2\,g^3}\cdot\frac14\sqrt{\frac{\pi}{g}}
= \frac{16^g(2g)!}{256\,\sqrt{2\pi}\,g^{7/2}} .
\]

\paragraph{Numerical evidence.}
By Theorem~\ref{thm:gelfand-reduction}, Conjecture~\ref{conj:single-curve} reduces to the single limit $N\rho_U \to 2$, and $\rho_U$ can be computed directly. Call $\sigma \in C_\sigma$ a \emph{lift} of a matching $\beta_*$ if $\sigma^2 = \beta_*$; then $\rho_U$ is the proportion of lifts of a fixed $\beta_*$ of coset type $(n)$ for which $\sigma\alpha_0$ is an $N$-cycle. By Lemma~\ref{lem:squaring-surjective}, a lift amounts to a grouping of the $4g-2$ transpositions of $\beta_*$ into pairs together with the choice of one of the two square roots of each pair, so lifts are easy to list for small $g$ and to sample uniformly for large $g$. We have computed $\rho_U$ exactly for $g \leq 4$, by running through all lifts, and estimated it by Monte Carlo sampling for selected genera $5 \leq g \leq 500$, that is, for $N$ up to about $4000$; the sampler and its validation are described in Appendix~\ref{app:sampler}. Table~\ref{tab:rhoU-data} gives the results for $103 \leq g \leq 500$, the range that bears on the limit, together with the corresponding values of $\sqrt N\,P_g = \sqrt N\,q_g\rho_U/p_g$. Since $q_g$ and $p_g$ are known exactly, this second column carries the same information as the first, expressed through the probability that appears in Conjecture~\ref{conj:single-curve}.

\begin{table}[H]
\centering
\renewcommand{\arraystretch}{1.15}
\begin{tabular}{@{}r r l c c@{}}
\toprule
$g$ & $N$ & $M$ & $N\rho_U$ & $\sqrt N\,P_g$\\
\midrule
$103$ & $820$  & $2.5\times10^{7}$ & $1.987(8)$  & $1.2497(51)$\\
$130$ & $1036$ & $1.5\times10^{7}$ & $1.991(12)$ & $1.2514(74)$\\
$160$ & $1276$ & $1.5\times10^{7}$ & $2.001(13)$ & $1.2568(82)$\\
$200$ & $1596$ & $1.8\times10^{7}$ & $2.004(13)$ & $1.2582(84)$\\
$250$ & $1996$ & $1.5\times10^{7}$ & $1.990(16)$ & $1.2489(102)$\\
$300$ & $2396$ & $1.6\times10^{7}$ & $2.008(17)$ & $1.2597(109)$\\
$400$ & $3196$ & $1.2\times10^{7}$ & $2.049(23)$ & $1.2853(146)$\\
$500$ & $3996$ & $1.1\times10^{7}$ & $1.957(27)$ & $1.2273(167)$\\
\bottomrule
\end{tabular}
\caption{Monte Carlo values of $N\rho_U$ and $\sqrt N\,P_g$ for $103 \leq g \leq 500$. $M$ is the number of sampled lifts; a parenthesized figure is one standard error in units of the last displayed digit. The conjectured limits are $N\rho_U\to2$ and $\sqrt N\,P_g\to\sqrt{\pi/2}=1.25331$. The values for smaller genera, including the exact values for $g \leq 4$, and all raw counts and random seeds are included in the ancillary files.}
\label{tab:rhoU-data}
\end{table}

In Table~\ref{tab:rhoU-data}, the estimates of $N\rho_U$ lie on both sides of $2$, between $1.957$ and $2.049$, with no visible trend: all but one are within $1.7$ standard errors of $2$, the remaining one, at $g = 400$, within $2.1$, and their weighted mean is $1.995$ with standard error $0.005$. Correspondingly, $\sqrt N\,P_g$ lies between $1.227$ and $1.285$, against the conjectured limit $\sqrt{\pi/2} = 1.2533$, and is constant to within the statistical error as $N$ grows from about $800$ to $4000$, so the data give no indication that $P_g$ decays at a rate other than $N^{-1/2}$. At smaller genera $N\rho_U$ approaches $2$ from below, slowly: the deficit is still about $0.03$ at $g = 53$ and $g = 86$. The data thus agree with Conjecture~\ref{conj:single-curve} to within about half a percent, but they cannot exclude a limit of $N\rho_U$ slightly below $2$.

\paragraph{Concluding remarks.}
Theorem~\ref{thm:gelfand-reduction} reduces the single-curve problem to the asymptotics of one quantity, the one-face probability $\rho_U$. The character-theoretic method of Section~\ref{sec:probability} does not carry over: since the Petrie condition involves $\sigma^2$ as well as $\sigma$, the corresponding count leads to a sum over triples of irreducible characters of $S_N$, of which only one collapses to hooks on the $N$-cycle, and known results on characters of $S_N$ provide no way to evaluate or estimate the remaining double sum over partitions. The reduction of Theorem~\ref{thm:gelfand-reduction} avoids this sum, and proving Conjecture~\ref{conj:single-curve} amounts to a direct analysis of $\rho_U$, the probability that a uniformly random lift of $\beta_*$ is unicellular. We hope to return to this question in future work.

\appendix
\section{Algorithms}\label{app:algorithms}

\subsection{Exact enumeration of rooted maps}\label{sec:algorithm}

Fix $\varphi = (1\,2\,\cdots\,N)$. A rooted graph is a fixed-point-free involution $\alpha$ with $\sigma := \varphi\alpha \in C_\sigma$, and the isomorphism classes are the $\langle\varphi\rangle$-orbits of such $\alpha$. Setting $\alpha(i) = j$ determines $\sigma(i) = \varphi(j)$ and $\sigma(j) = \varphi(i)$, so as $\alpha$ is built one transposition at a time, $\sigma$ accumulates as a partial permutation whose chains can be tracked incrementally; a branch is abandoned as soon as a chain would exceed length $4$ or close at a length other than $4$.

\begin{algorithm}[ht]
\caption{Enumerate rooted unicellular $4$-regular ribbon graphs of genus $g$}\label{alg:backtrack}
\begin{algorithmic}[1]
\State $N \gets 8g - 4$; \ $\varphi(i) := (i \bmod N) + 1$; \ $\mathcal{S} \gets \emptyset$.
\State $\alpha[i] \gets {\perp}$, $\sigma[i] \gets {\perp}$ for all $i$.
\Procedure{Extend}{}
  \State Let $i$ be the smallest index with $\alpha[i] = {\perp}$.
  \If{no such $i$} append $\alpha$ to $\mathcal{S}$ and return. \EndIf
  \For{each $j > i$ with $\alpha[j] = {\perp}$}
    \State set $\alpha[i]\gets j,\ \alpha[j]\gets i,\ \sigma[i]\gets\varphi(j),\ \sigma[j]\gets\varphi(i)$.
    \If{the $\sigma$-chains through $i$ and $j$ are valid (open of length $\leq 4$, or closed of length $4$)}
      \State \textsc{Extend}.
    \EndIf
    \State undo the assignment.
  \EndFor
\EndProcedure
\State call \textsc{Extend}; \Return $\mathcal{S}$.
\end{algorithmic}
\end{algorithm}

\begin{proposition}\label{prop:alg-correct}
Algorithm~\ref{alg:backtrack} terminates and returns exactly the set of fixed-point-free involutions $\alpha$ with $\varphi\alpha \in C_\sigma$.
\end{proposition}

\begin{proof}
The recursion assigns each index once, so it visits finitely many partial matchings. The pruning conditions are necessary for $\sigma \in C_\sigma$ at every partial stage (a $\sigma$-chain that has exceeded length $4$, or closed at length $\neq 4$, can never become a union of $4$-cycles), so no valid $\alpha$ is discarded; at a leaf every index is matched and every $\sigma$-chain is a closed $4$-cycle, so $\sigma \in C_\sigma$, and conversely every such $\alpha$ is reached.
\end{proof}

Conjugation by $\langle\varphi\rangle$ then partitions the output $\mathcal{S}$ into orbits: their number is $\mathcal{N}_g$, each orbit's stabilizer order is $|\Aut(G)|$ (with $\sum_{[G]} N/|\Aut(G)| = |\mathcal{S}| = \mathcal{R}_g$ as a consistency check), and the Petrie cycle type of a representative gives the component count. The cost is governed by $\mathcal{R}_g$: $\mathcal{R}_4 \approx 4.7\times 10^6$ and $\mathcal{R}_{10} \approx 7.6\times 10^{26}$, so listing all rooted maps is practical only for $g \leq 3$. For $g = 4$ we count the classes with Burnside's lemma instead,
\[
\mathcal{N}_g = \frac1N\sum_{r=0}^{N-1}\bigl|\mathrm{Fix}(\varphi^r)\bigr|,
\]
where $\mathrm{Fix}(\varphi^r)$ is the set of rooted maps $\alpha$ commuting with $\varphi^r$. The term $r = 0$ is $\mathcal{R}_g$, known exactly from the character sum, and the other terms depend only on $\gcd(r,N)$; the rooted maps invariant under a nontrivial rotation are few, and are listed by the same search with the invariance imposed. This gives $\mathcal{N}_4 = 169{,}772$. For the rooted count alone the character sum is superior at every genus, and the enumeration is used only for the data it does not give, namely the number of classes, the automorphism orders and the component counts.

\subsection{Monte Carlo estimation of $\rho_U$}\label{app:sampler}

Fix the half-edge set $[N]=\{0,1,\dots,N-1\}$, the reference matching $\alpha_0(i) = i\oplus1$ ($i\oplus1$ denoting the toggle of the lowest bit, so $\alpha_0$ swaps $2k\leftrightarrow 2k+1$), and the matching $\beta_* := \prod_{t=0}^{N/2-1}\bigl(2t+1,\ (2t+2)\bmod N\bigr)$, the ``odd'' edges $\{1,2\},\{3,4\},\dots,\{N-1,0\}$ of the $N$-gon on $0,1,\dots,N-1$. Then $\alpha_0\cup\beta_*$ is a single $N$-cycle, so $\beta_*$ has coset type $(n)$ and $\rho(\beta_*) = \rho_U$ (Lemma~\ref{lem:rho-invariant}); any other $\beta_*$ of coset type $(n)$ would give the same value. By the proof of Lemma~\ref{lem:squaring-surjective}, the $(4g-2)!/(2g-1)!$ lifts of $\beta_*$ are parametrized bijectively by (a) a pairing of the $4g-2$ transpositions of $\beta_*$ into $2g-1$ unordered pairs, and (b) for each pair $\{(x\,y),(u\,v)\}$ a choice between the two $4$-cycles $(x\,u\,y\,v)$ and $(x\,v\,y\,u)$ squaring to $(x\,y)(u\,v)$. Sampling (a) uniformly (a Fisher--Yates shuffle of the transposition list followed by consecutive pairing) and (b) by independent fair bits therefore produces a uniform lift, and $\rho_U = \Pr_\sigma[\sigma\alpha_0\in C_N]$ is estimated by the proportion $\hat\rho_U = S/M$ of $M$ draws for which the orbit of the half-edge $0$ under $\varphi(i) = \sigma(i\oplus1)$ has length exactly $N$ (Algorithm~\ref{alg:rhoU}). The standard error is $\mathrm{se}(\hat\rho_U) = \sqrt{\hat\rho_U(1-\hat\rho_U)/M}$, a relative error $\approx S^{-1/2}$, and the relative error of $\hat P_g = q_g\hat\rho_U/p_g$ equals that of $\hat\rho_U$ since $q_g$ and $p_g$ are exact. Sampling $\rho_U$ rather than $P_g$ directly is what makes large $g$ accessible: the joint event $\{\varphi\in C_N,\ \tau\in C_{(N/2)^2}\}$ has probability $q_g\rho_U$ against $\rho_U$ for the one-face event among lifts, so a direct estimate of $P_g$ to the same relative accuracy would need a factor $1/q_g \sim \sqrt{2N/\pi}$ more draws---about $20$ at $g = 100$ and $50$ at $g = 500$. Each draw costs $O(N)$ time and space.

The C implementations of Algorithms~\ref{alg:backtrack} and~\ref{alg:rhoU} (with a \texttt{xoshiro256**} pseudo-random generator), the programs computing $\rho_U$ exactly for $g \leq 4$ and $\mathcal{N}_4$ by Burnside's lemma, scripts evaluating $q_g$ and $p_g$ and assembling Table~\ref{tab:rhoU-data}, and the raw counts together with the random seeds used are provided as ancillary files to the arXiv version of this paper.

\begin{algorithm}[ht]
\caption{Monte Carlo estimate of $\rho_U$ at genus $g$ (one run of $M$ draws)}\label{alg:rhoU}
\begin{algorithmic}[1]
\State $N\gets 8g-4$,\quad $n\gets N/2$,\quad $S\gets0$
\State define $\alpha_0$ by $\alpha_0(i)=i\oplus1$ \Comment{pairs $2k\leftrightarrow 2k{+}1$}
\State define $T[t]=\{\,2t{+}1,\ (2t{+}2)\bmod N\,\}$ for $t=0,\dots,n{-}1$ \Comment{edges of $\beta_*$}
\For{$s=1,\dots,M$}
  \State draw a uniform permutation $\pi$ of $\{0,\dots,n{-}1\}$ \Comment{Fisher--Yates shuffle}
  \For{$q=0,\dots,\tfrac n2-1$} \Comment{$\tfrac n2=2g-1$ four-cycles}
    \State $\{x,y\}\gets T[\pi[2q]]$,\quad $\{u,v\}\gets T[\pi[2q{+}1]]$
    \State draw a uniform bit $b\in\{0,1\}$
    \If{$b=1$}
      \State set $\sigma(x){=}u,\ \sigma(u){=}y,\ \sigma(y){=}v,\ \sigma(v){=}x$ \Comment{$4$-cycle $(x\,u\,y\,v)$}
    \Else
      \State set $\sigma(x){=}v,\ \sigma(v){=}y,\ \sigma(y){=}u,\ \sigma(u){=}x$ \Comment{$4$-cycle $(x\,v\,y\,u)$}
    \EndIf
  \EndFor
  \State $c\gets0$,\quad $j\gets0$
  \Repeat
    \State $j\gets\sigma(j\oplus1)$,\quad $c\gets c+1$ \Comment{advance one step of $\varphi=\sigma\alpha_0$}
  \Until{$j=0$}
  \If{$c=N$} \State $S\gets S+1$ \Comment{single face}
  \EndIf
\EndFor
\State \Return $\hat\rho_U=S/M$ \ and \ $\mathrm{se}=\sqrt{\hat\rho_U(1-\hat\rho_U)/M}$
\end{algorithmic}
\end{algorithm}

\paragraph{Validation.}
(i) At $g=2,3,4$, runs of $2\times10^7$ draws return $\hat\rho_U = 0.09995$, $0.07140$, $0.05598$ against the exact $\tfrac1{10}$, $\tfrac1{14}$, $\tfrac{37}{660}$, within $0.7$, $0.6$ and $1.7$ standard errors, and assembling with the exact $q_g$, $p_g$ reproduces the exact values of $P_g$ to the same accuracy; this checks the pipeline end to end where the truth is known. (ii) The code has no genus-dependent branch---only the array length $N$ and the loop bounds vary with $g$---and uses $64$-bit counters, which precludes overflow at the sizes used ($N\le3996$, $M\le2.5\times10^7$). (iii) Two invariants holding with probability one, $\sgn(\sigma\alpha_0) = -1$ (Lemma~\ref{lem:parity}) and the doubled cycle type of $\sigma^2\alpha_0$ (Lemma~\ref{lem:tau-doubled}), can be checked on every sample; they held on all $6\times10^7$ samples of the runs in (i). (iv) Estimating $P_g$ directly, by sampling $\sigma$ uniformly on $C_\sigma$ and testing both $\varphi\in C_N$ and $\tau\in C_{(N/2)^2}$, agrees with the exact values for $g = 2,3,4$ and with the reduction-based estimate for $g = 5$, within $1.4$ standard errors. (v) At $g = 2$ the exact values are consistent with the enumeration of Section~\ref{sec:lowgenus}: $P_2 = \tfrac{4}{15}$ and $A_2/12! = \tfrac{15}{4}$ give $A_2^{(1)}/12! = 1$, i.e.\ the unique single-curve class has trivial automorphism group, as the enumeration confirms. The sampled lifts are a negligible fraction of all lifts (about $10^{-501}$ at $g = 103$); the accuracy is governed by the number $S$ of one-face hits, as for any Monte Carlo estimate of a mean.

\end{document}